\documentclass[a4paper,reqno]{amsart}                 
\usepackage{mathptmx}
\usepackage{a4wide,amsmath,amssymb,hyperref}

\usepackage{dsfont,cite}              

\newtheorem{theorem}{Theorem}[section]
\newtheorem{lemma}[theorem]{Lemma}
\newtheorem{corollary}[theorem]{Corollary}
\newtheorem{proposition}[theorem]{Proposition}
\newtheorem{definition}[theorem]{Definition}
\newtheorem{example}[theorem]{Example}
\newtheorem{remark}[theorem]{Remark}

\numberwithin{equation}{section}

\def\bZ{{\mathbb Z}}
\def\bE{{\mathbb E}}
\def\bN{{\mathbb N}}

\def\bC{{\mathbb C}}
\def\bR{{\mathbb R}}

\def\Q{{\mathsf{Q}}}        

\def\cD{\mathcal{D}}
\def\cE{\mathcal{E}}
\def\cF{\mathcal{F}}

\def\cO{\mathcal{O}}
\def\cP{\mathcal{P}}
\def\cQ{\mathcal{Q}}

\def\cS{\mathcal{S}}

\def\cZ{\mathcal{Z}}

\def\supp{{\operatorname{supp}}}
\def\R{\mathbb{R}}

\newcommand{\im}{\operatorname{i}}

\begin{document}

\title[Homogeneous Mixed-norm Triebel--Lizorkin spaces]
{Wavelet transforms for Homogeneous\\ Mixed-norm Triebel--Lizorkin spaces}

\author{A.~G.~Georgiadis}
\author{J.~Johnsen}
\author{M.~Nielsen}
\address{Department of Mathematical Sciences, Aalborg University, Fredrik Bajers Vej 7G, DK-9220 Aalborg \O st, Denmark}
\email{nasos@math.aau.dk}
\email{jjohnsen@math.aau.dk}
\email{mnielsen@math.aau.dk}
\thanks{The authors are supported by the Danish 
Research Council, Natural Sciences grant no.~0602-02505B, the second author also partially by 
grant no.~4181-00042\\[4\jot]
{\tt To appear in Monatshefte f\"ur Mathematik}%
}

\begin{abstract}
  Homogeneous mixed-norm Triebel--Lizorkin spaces are introduced and studied with the use of a
  discrete wavelet transformation, the so-called $\varphi$-transform. This extends the  classical
  $\varphi$-transform approach introduced by Frazier and Jawerth to the setting of mixed-norm
  spaces. Moreover, the theory of the $\varphi$-transform is enhanced through a precise
  definition of the synthesis operator, in terms of a Pettis integral, and a number of rigorous
  results for this operator.
  Especially its terms can always be summed in any order, without changing the resulting distribution.
\end{abstract}

\maketitle

\section{Introduction}\label{Introduction}
\enlargethispage{2\baselineskip}\thispagestyle{empty}      

One of the central problems in harmonic analysis is to estimate the norm of a distribution in a
smoothness space by the norm of a related sequence, through a discrete representation. Perhaps the
most well-known example of this is Parseval's identity connecting the $L_2$-norm of a function with the
$\ell^2$-norm of the sequence of its Fourier coefficients.  

The $\varphi$-transform has been systematically exploited to obtain such discrete
representations since the celebrated papers of Frazier and Jawerth
\cite{MR808825,MR1070037}. Using the $\varphi$-transform, Frazier and Jawerth explored the deeper properties of
the homogeneous Triebel--Lizorkin spaces $\dot{F}^{s}_{pq}$, and this led to a wealth of subsequent
work by many authors. For homogeneous Triebel--Lizorkin spaces on $\mathbb{R}^n$ we may refer the reader to
works of Bownik and Ho and of Kyriazis and
Petrushev~\cite{MR2186983,MR1981430,MR1862566,MR2204289}, and for anisotropic decompositions  
to papers of Borup and Nielsen and  of Bownik and Ho~\cite{MR2296727,MR2423282,MR2186983}. 
Bownik treated non-diagonal dilations and doubling measures \cite{MR2358763}.  
For spaces on other domains such as the sphere, see for example
\cite{MR2943727,MR3271256,MR1981430,MR2548033,MR2439670}. 

In the past decade there has been an interest in analysing regularity of
functions by means of inhomogeneous \emph{mixed norm} Triebel--Lizorkin spaces
$F^s_{\vec{p},q}$. This is a way to measure the degree of smoothness $s$
as well as integrability $\vec{p}=(p_1,\dots,p_n)$ with different
integral exponents in different directions, and a certain microscopic
parameter $q$, in an efficacious environment of harmonic analysis.

For contributions on the embedding properties and traces on hyperplanes, the reader is referred to works of
Johnsen and Sickel~\cite{MR2319603,MR2401510},
who in collaboration with Munch~Hansen also analysed embeddings and equivalent characterisations of such
mixed-norm spaces, their invariance under coordinate transformations and traces on hyperplanes and domains;
cf.\ \cite{MR3073635,MR3178988,MR3377120}. Continuity of pseudo-differential operators 
in this set-up has been treated by Georgiadis and Nielsen~\cite{MR3573690}.

In this paper we take up the construction of wavelet bases for the mixed norm
Triebel--Lizorkin spaces. This has seemingly not been done in the 
context before, perhaps because of certain difficulties in handling the mixed norms. 
Though for the basic mixed-norm Lebesgue spaces $L_{\vec p}(\bR^n)$, 
there is a recent construction for $n=2$ in Section~6 of the work of Torres and Ward~\cite{MR3393695}.
Anyhow, the wavelets seem useful for implementation of the $F^s_{\vec{p},q}$-spaces in most
applied branches of mathematics. 

In our treatment of wavelets, we use the convenient approach of Frazier and Jawerth
by adapting their $\varphi$-transform from \cite{MR808825,MR1070037}, which to some extent bridges
the gap between harmonic analysis and the general multiresolution analysis within wavelet theory. To utilise the $\varphi$-transform
at its best, we introduce the \emph{homogeneous} mixed-norm Triebel--Lizorkin spaces
$\dot{F}^{s}_{\vec{p},q}$ with $\vec{p}=(p_1,\dots,p_n)$ for
$0<p_j<\infty$, $s \in\R$ and $0<q \le \infty$ (and recall in passing that \cite{MR808825}
described a straightforward way to carry over many of the obtained results to corresponding inhomogenenous
spaces, that have a single low-pass filter).
We also develop the basic theory of the spaces $\dot{F}^{s}_{\vec{p},q}$ along the way.
We note that in a recent work Hart, Torres and Wu~\cite{HTW} have applied the wavelet decomposition of 
$\dot F^s_{\vec p,q}$ obtained in the present paper.

Following Frazier and Jawerth, we define corresponding sequence spaces $\dot{f}^{s}_{\vec{p},q}$
and introduce two operators, depending on some admissible functions $\varphi,\psi$:
\begin{align*}
   S_{\varphi}&\colon \dot{F}^{s}_{\vec{p},q}\rightarrow\dot{f}^{s}_{\vec{p},q}\quad(\varphi\text{-transform})
\\[2\jot]
   T_{\psi}&\colon\dot{f}^{s}_{\vec{p},q}\rightarrow\dot{F}^{s}_{\vec{p},q}\quad (\text{``inverse'' $\varphi$-transform})
\end{align*}
Hereby $S_\varphi$ maps a function $f$ to its wavelet coefficients (sends signals to sequences),
whereas $T_\psi$ transforms sequences to functions.

The main result of this article is the following variant of the classical
$\varphi$-transform result, which
states inter alia that every function $f$ in $\dot{F}^{s}_{\vec{p},q}$ will be
reconstructed by using $T_\psi$ on its wavelet coefficients $S_\varphi f$:

\begin{theorem}
The above linear transformations $S_{\varphi}$ and $T_{\psi}$ are bounded
operators, and $T_{\psi}(S_{\varphi}f)=f$ holds for every  $f\in\dot{F}^{s}_{\vec{p},q}$.
\end{theorem}

A convenient consequence of the induced formula $T_{\psi}\circ S_{\varphi}=I$
is that completeness of $\dot{F}^{s}_{\vec{p},q}$ follows at once from that of the simpler sequence
space $\dot{f}^{s}_{\vec{p},q}$.

Our treatment departs from standard formulas for $S_\varphi$ and $T_\psi$, namely
\begin{align} \label{Sintro}
  (S_\varphi f)_Q&=\langle f,\varphi_Q\rangle \vphantom{\sum}\hphantom{\psi_Q(x)}\qquad\text{for each dyadic cube $Q$},
\\  \label{Tintro}
     T_\psi a(x)&=\sum_{Q\in\cQ}a_Q\psi_Q(x)\qquad\text{for each sequence $a=\{a_Q\}_{Q\in\cQ}$}.
\end{align}
However, while $S_\varphi f$ is well defined for $f\in\cS'/\cP$, 
the operator $T_\psi$ is a more delicate object since the sum in \eqref{Tintro}, and hence $T_\psi$
itself, only makes sense a priori on sequences $\{a_Q\}$ of finite support.

The synthesis operator  $T_\psi$ has been further studied in the homogeneous set-up in by e.g.\
Kyriazis~\cite{MR1981430}, later by Bownik and Ho~\cite{MR2186983} and Bownik~\cite{MR2358763}, who
partially resolved the question of interpretation of \eqref{Tintro}.
Here we would like to present a new perspective on $T_\psi$ and put
the study of it in a rigorous framework. 

So let us recall that previously boundedness of $T_\psi$ on sequences of finite support
has been followed up by extension by continuity---that for $q=\infty$ is inadequate due to a lack
of density. And often this extension $\widetilde T_\psi$ has entered composition formulas, like
$T_\psi\circ S_\varphi=I$, in a heuristic way with a tacit assumption that also $\widetilde T_\psi$ acts as in
\eqref{Tintro}---although the notation $\sum_{Q\in\cQ}$ was never explicitly assigned any specific meaning.
Indeed, the set of dyadic cubes, $\cQ$, can be numbered in many ways, so some condition of \emph{integrability} must
be imposed on $\{a_Q\}$ to get a consistent theory. 
To our knowledge, neither the foundational papers by Frazier and Jawerth~\cite{MR808825,MR1070037} nor the subsequent
literature have explicitly addressed this integrability.

On these grounds it seems well motivated that we here revise the foundation of the synthesis
operator $T_\psi$ by suggesting a concise definition. Specifically we define 
\begin{equation} \label{Pettis-intro}
  T_\psi a=\int_{\cQ} a_Q\psi_Q\,d\tau_{1+n}  
\end{equation} 
in terms of a Pettis integral (or weak Bochner integral) over the set of dyadic cubes $\cQ\simeq \bZ\times\bZ^n$
with respect to the counting measure $\tau_{1+n}$. This integral is the element  in $\cZ'=\cS'/\cP$
that fulfils
  \begin{equation} \label{Pettis-intro'}
    \langle T_\psi a,\phi\rangle =\int_{\cQ} \langle a_Q\psi_Q,\phi\rangle\,d\tau_{1+n} \qquad\text{for all $\phi\in Z$}.
  \end{equation}
From this explicit definition we
\begin{itemize} \addtolength{\itemsep}{\jot}
\item obtain that $R(S_\varphi)\subset D(T_\psi)$, that is, the above $T_\psi$ is defined on the
  entire range of the wavelet transform $S_\varphi$;
\item show explicitly for our sequence spaces that
  $\dot f^s_{\vec p,q}\subset D(T_\psi)$;
\item rigorously prove that $T_\psi(S_\varphi f)=f$ for every $f\in\cS'/\cP$, when
  $\varphi$, $\psi$ are admissible test functions  satisfying the well-known \emph{reconstruction} identity,
  \begin{equation}  \label{LP-condition}
    \sum_{\nu=-\infty}^\infty \overline{\hat\varphi(2^{-\nu}\xi)}\hat\psi(2^{-\nu}\xi)=1 
    \quad\text{ for $\xi\ne0$};    
  \end{equation}
\item deduce from \eqref{LP-condition} that the following properties are equivalent:
\begin{itemize}
\item[$\diamond$] $P=S_\varphi\circ T_\psi$ is a projection for which
  $P=I$, i.e.\ $S_\varphi(T_\psi a)=a$ for every $a\in D(T_\psi)$,
\item[$\diamond$] $\varphi$, $\psi$ fulfil  the \emph{biorthogonality} condition 
  \begin{equation}  \label{biortho-intro}
    \int_{\bR^n}\psi_Q\overline{\varphi_J}\,dx =\delta_{Q,J} \qquad\text{(Kronecker delta)};
  \end{equation}
\item[$\diamond$] there exists a numbering of the dyadic cubes such that $\{\psi_{Q_j}\}_{j\in\bN}$ is a basis for
  $\cS'/\cP$,
  \item[$\diamond$]  \emph{every} numbering of the $\{\psi_Q\}_{Q\in\cQ}$
  induces an \emph{unconditional} basis for $\cS'/\cP$.
\end{itemize}
\end{itemize}
The last point has a striking consequence for the entire $\varphi$-transform set-up. Indeed, one
could hope to fix a numbering 
such that $f=\sum_{j=1}^\infty c_j\psi_{Q_j}$ would hold for any $f\in\cS'/\cP$ for unique scalars $c_j$;
then stick to this numbering, and term $\{c_j\}$ the wavelet coefficients of $f$ etc. 
But, in this case, $\{\psi_{Q_j}\}_{j\in\bN}$ would be a basis for $\cS'/\cP$, and
for this it is \emph{necessary} that $\varphi$, $\psi$ fulfil the biorthogonality condition
\eqref{biortho-intro}; cf.\ the above or Theorem~\ref{basis-thm} below. 
And in case $\varphi$, $\psi$ do fulfil \eqref{biortho-intro}, this basis would be unconditional by
the last part above, so the summation order would be inconsequential. 

But when $\varphi$, $\psi $ violate the biorthogonality \eqref{biortho-intro},
then no numbering yields a basis of $\cS'/\cP$ and one has $P=S_\varphi T_\psi\ne I$, so it is clear that 
$T_\psi$ is not injective (as $S_\varphi$ is). Hence in this case $f$ will always have other possible wavelet coefficients than
the canonical ones in $S_\varphi f$. 

However, even without biorthogonality all summation orders in $T_\psi a$ yield
the same result---this seems to be a new observation, cf.\  the below Corollary~\ref{Tpsi-cor}. 
It follows from the w$^*$-approximation property of $T_\psi$ 
in Theorem~\ref{Tpsi-thm} as a  \emph{post festum} result of the definition in
\eqref{Pettis-intro}--\eqref{Pettis-intro'}.
Adding this conclusion to the list we get:
\begin{itemize}
  \item   For $T_\psi$ in \eqref{Pettis-intro} the value on each sequence $a$ in
  $D(T_\psi)$ is obtained as $T_\psi a=\sum_{j=1}^\infty a_{Q_j}\psi_{Q_j}$ by summing the terms
  as an infinite series in $\cZ'=\cS'/\cP$ in any order.
\end{itemize}
Thus \hskip-.15pt the \hskip-.15pt dichotomy \hskip-.15pt given \hskip-.15pt by \hskip-.15pt
{\small(}non-\hskip-.4pt{\small)}biorthogonality 
\hskip-.15pt is \hskip-.15pt circumvented \hskip-.15pt by \hskip-.15pt using \hskip-.15pt the \hskip-.15pt 
Pettis \hskip-.15pt integral \hskip-.15pt
\eqref{Pettis-intro} \hskip-.15pt for $T_{\hskip-.2pt \psi}$. 

Furthermore, it seems to be an open question whether biorthogonal wavelet systems exist within
the $\varphi$-transform framework; cf.\ Remark~\ref{Meyer-rem} below. 
Although we have strived for minimal requirements on all constants,
extending admissibility of $\varphi$, $\psi$ from that in \cite{MR808825,MR1070037}, 
it is not possible to obtain orthonormal wavelets for $\varphi=\psi$ in our set-up; cf.\ the classical consideration in
Remark~\ref{orto-rem} below. Thus Meyer's orthonormal wavelets, which are briefly recalled in
Example~\ref{Meyer}, fall outside the (present) theory of the $\varphi$-transform.

\bigskip
Alongside this rigorous definition of $T_\psi$, we have also worked out a precise version
of  Peetre's homogeneous Littlewood--Paley decomposition; cf.\ Appendix~\ref{LP'-app} below. 
As this corrects the previous literature in two ways, and enters our proof of the formula $T_\psi\circ
S_\varphi=I$, we review it here:

When $\hat\phi\in\cS$ is supported in an annulus $0<C_0\le|\xi|\le C^0$ and 
$1=\sum_{\nu\in\bZ} \hat\phi(2^{-\nu}\xi)$ for $\xi\ne0$, then every $f\in\cS'$ has a
homogeneous Littlewood--Paley decomposition with an \emph{explicit} asymptotic behaviour for
$\nu\to-\infty$. 

Namely, by working with polynomial corrections $P_{m,N}$ with $N\in\bN$ and a fixed
degree $m\ge-1$, one has
\begin{equation}
  \label{eq:LPintro}
  f(x)=\sum_{\nu=-N}^\infty \phi(2^{-\nu}\cdot)*f(x) + P_{m,N}(x) + R_m(x).
\end{equation}
Here the remainder
term fulfils $R_m=\cO(2^{-N(n+m+1-d)})$ in $\cS'$-seminorm, $d$ being the $\cS'$-order of $f$. So
for degrees $m\ge d-n$, clearly $R_m\to0$ exponentially fast for $N\to\infty$, whence $f$ is 
well represented in $\cS'$ for $N\to\infty$ by the band-limited series $\sum_{\nu\ge-N} \phi(2^{-\nu}\cdot)*f$
corrected by $P_{m,N}$.

As a novelty the $P_{m,N}$ are \emph{uniquely} determined (asymptotically for
$N\to\infty$) as the degree $m$ Taylor polynomials at $x=0$ of a convolution
$2^{-nN}\Phi(2^{-N}\cdot)*f$; cf.\ \eqref{fPhi-id}. The correcting $P_{m,N}$ can moreover be
omitted ($m=-1$) for distributions having $d<n$, in view
of the estimate of $R_m$.

Previously the literature has indicated, through several contributions, that in general one would
meet arbitrary polynomials $P$ and $P_N$ on the left- and right-hand sides of \eqref{eq:LPintro}, whilst a few authors have
claimed some restrictions for the degrees of $P$, $P_N$; cf.\ Remark~\ref{Peetre-rem} below.

But this picture is misleading in two ways:
our analysis shows that such $P$, $P_N$ must be interrelated, as $P-P_N$ asymptotically equals the Taylor polynomial
$P_{m,N}$ due to the uniqueness---and on the contrary the degree $m\ge-1$ can be arbitrary. 
Our remainder estimate, which seems to be new in itself, shows that even by
omitting polynomials, $f$ will have a specific asymptotic representation by $\sum_{\nu\in\bZ}
\phi(2^{-\nu}\cdot)*f$, which e.g.\ is exact for $d<n$ and, because of \eqref{eq:LPintro},
has an error for $d\ge n$ with the leading term given by $P_{d-n,N}$.

This improved insight results at once from a general analysis of the ``wrong'' limit $t\to0^+$ of
convolutions of the form
\begin{equation}
  t^n\Phi(t\cdot)* f(x),\qquad \Phi\in\cS, f\in \cS'.
\end{equation}
For details on how such convolutions behave asymptotically in $\cS'$ as their Taylor
polynomial $P_m$ of degree $m$ at $x=0$, the reader is referred to our analysis in
Proposition~\ref{asymp-prop}.

For the wavelet reconstruction formula $f=T_\psi(S_\varphi f)$ the consequences of the above are
immediate, because the right-hand side $T_\psi(S_\varphi f)$ identifies with the series in \eqref{eq:LPintro} for a
special choice of $\phi$. Thus one always has that $f=T_\psi(S_\varphi f)$ in the quotient space $\cS'\setminus\cP$
(if \eqref{LP-condition} holds), but it even holds in $\cS'$ itself for all  $f$ having $d<n$. In
general the above shows which polynomials to add and how fast the wavelet reconstruction converges.

\subsection*{Contents} Preliminaries and notation are summed up in
Section~\ref{Preliminaries}. General results for the synthesis
operator $T_\psi$ defined by the Pettis integral are developed in
Section~\ref{transform-sect}.
Triebel--Lizorkin spaces with mixed norms are introduced and studied in
Section~\ref{MTL-sect} together with the corresponding sequence spaces.
Section~\ref{TSI-sect} is devoted to our results on $S_\varphi$
and $T_\psi$ in the scales of mixed-norm Triebel--Lizorkin spaces.
Some technical proofs are given in Appendix~\ref{p*p-app}--\ref{LP'-app}.

\section{Preliminaries}\label{Preliminaries}
\subsection{Notions and notation} 
Generally we follow the notation 
of H{\"o}rmander~\cite{MR1996773} for the Fourier transformation 
$\hat f(\xi)=\cF f(\xi)=\int e^{-\operatorname{i} x\cdot\xi}f(x)\,dx$ and the distribution spaces, so $\cS'=\cS'(\bR^n)$ and
$\cD'=\cD'(\bR^n)$ are dual to the spaces of Schwartz functions $\cS=\cS(\bR^n)$ and smooth
functions of compact support $C_0^\infty(\bR^n)$, respectively. In particular
$D^\alpha=(-\operatorname{i})^{|\alpha|}\partial^\alpha$ for each multiindex $\alpha$. However,
we use for convenience bracket notation and take functionals to be anti-linear (unless it is stated
otherwise), so if e.g.\ $f$ is locally integrable,
$\langle f,\varphi\rangle=\int f(x)\overline{\varphi(x)}\,dx$ for $\varphi\in C_0^\infty$.

If $\vec{p}=(p_1,\dots,p_n)$ with $0<p_1,\dots,p_n<\infty$ a function $f\colon\mathbb{R}^n\rightarrow \mathbb{C}$
belongs to $L_{\vec{p}}=L_{\vec{p}}(\mathbb{R}^n)$ if
\begin{equation}  \label{Lp}
  \|f\|_{\vec{p}}:=
  \Big(\int_\mathbb{R}\cdots\Big(\int_\mathbb{R}\Big(\int_\mathbb{R} 
  |f(x_1,\dots,x_n)|^{p_1} dx_1\Big)^{\frac{p_2}{p_1}} dx_2\Big)^{\frac{p_3}{p_2}}\cdots dx_n\Big)^{\frac{1}{p_n}}<\infty.
\end{equation}
Here $\|\cdot\|_{\vec{p}}$ is a quasi-norm, but
$(L_{\vec{p}},\|\cdot\|_{\vec{p}})$ is a Banach space if $\min(p_1,\dots,p_n)\geq 1$.
Of course, in general a change of integration order will lead to another iterated integral having another value.

Throughout we use the involution $\tilde\phi(x)=\overline{\phi(-x)}$, for which $\cF \tilde\varphi=\overline{\cF\varphi}$.
When a vector space $X$ has two equivalent quasi-norms $\|\cdot\|$ and $|\!|\!|\cdot|\!|\!|$, i.e.\
some numbers $c$, $C$ fulfil  $c\|x\|\le |\!|\!|x |\!|\!| \le C\|x\|$ for all $x\in X$, we
indicate this by writing $\|\cdot\|\approx |\!|\!|\cdot|\!|\!|$.

A topological vector space $E$ is said to have a sequence $\{x_n\}$ as basis when
each $x\in E$ can be written $x=\sum_{n=1}^\infty \lambda_n x_n$, with convergence in $E$, for a
unique sequence of scalars $\lambda_n$. Moreover, $\{x_n\}$ is called a Schauder basis if the linear
forms $x\mapsto \lambda_n(x)$ are continuous (this extension beyond the category of Banach spaces goes back at least to
Arsove and Edwards~\cite{MR0115068}). 
It is an unconditional basis of $E$ when, moreover,
$x=\sum_{n=1}^\infty \lambda_{p(n)} x_{p(n)}$ for any bijection $p\colon\bN\to\bN$.

Unimportant positive constants are denoted by $c$, although the value may depend on the place of
occurrence. As usual $t_+=\max(t,0)$, and $\mathds{1}_S$ stands for the characteristic function of
the set $S$. 

\subsection{The wavelet set-up}
Our basic building block is a function $\varphi\in\cS$ satisfying
\begin{equation}\label{phi2}
\supp\hat{\varphi}\subset\{\, \xi\in\mathbb{R}^n\mid K_0\le |\xi|\le K^0\,\},
\end{equation}
\begin{equation}\label{phi3}
|\hat{\varphi}(\xi)|\geq c>0 \quad \text{ for $K_1\le |\xi|\le K^1$}
\end{equation}
for some fixed constants $K_0<K_1<1<K^1<K^0$ (where superscripts refer to the upper bounds in
\eqref{phi2}--\eqref{phi3}) chosen so that
\begin{equation}
  \label{phi01}
  2K_1<K^1,\qquad K^0<\pi.
\end{equation}
The choice $K_0=1/2$, $K^0=2$ and $K_1=3/5$, $K^1=5/3$ was used in \cite{MR1070037}, but 
we extend the framework as described.

\begin{definition}\label{adm}
Functions $\varphi\in\cS$ satisfying \eqref{phi2}--\eqref{phi01} will be called admissible.
\end{definition}
Admissible functions obviously exist, but they are needed in pairs $(\varphi$, $\psi)$ that fulfil
the reconstruction identity in the following classical lemma:

\begin{lemma} \label{phi4-lem}
  To each admissible $\varphi$ there exist a function $\psi\in\cS$  which is admissible (for the
  same constants) and satisfies
  \begin{equation}  \label{phi4}
    \sum_{\nu\in \mathbb{Z}} \overline{\hat{\varphi}(2^{-\nu}\xi)}\hat{\psi}(2^{-\nu}\xi)=1 \qquad
    \text{for} \quad \xi\neq 0. 
  \end{equation}
\end{lemma}
We recall that $\hat\psi(\xi)=(h(\xi)-h(2\xi))/\overline{\hat\varphi(\xi)}$
reduces the claim in Lemma~\ref{phi4-lem} to a telescopic sum, if the auxiliary function $h\in
C_0^\infty(\bR^n)$ is chosen thus: when $|\hat\varphi(\xi)|>c/2$ for $|\xi|\in[\tilde K_1,\tilde
K^1]\subset \,]K_0,K^0[$ with $\tilde K_1<K_1$ and $K^1<\tilde K^1$, 
then $h(\xi)=0$ should hold for $|\xi|\ge \tilde K^1$ and $h(\xi)=1$ for 
$|\xi|\le 2\tilde K_1$ (where $2\tilde K_1<\tilde K^1$ holds by \eqref{phi01}), 
whilst $0<h<1$ elsewhere, for then $h(\xi)-h(2\xi)>0$ if and only if 
$|\xi|\in \,]\tilde K_1,\tilde K^1[\,$, so that \eqref{phi2} and \eqref{phi3} hold for $\hat\psi$.

We generally set $\varphi_{\nu}(x)=2^{\nu n}\varphi(2^{\nu}x)$ for $\nu\in\mathbb{Z}$, and any $\varphi\in\cS$.
For $\nu\in\mathbb{Z}$ and $k\in\mathbb{Z}^n$, we denote by $Q_{\nu k}$ the dyadic cube
\begin{equation*}
  Q_{\nu k}=\{(x_1,\dots,x_n)\in\mathbb{R}^n\mid k_i\leq 2^{\nu}x_i<k_i+1,i=1,\dots,n\}
\end{equation*}
and let $\mathcal{Q}$ be the set of these dyadic cubes; $Q$ will designate an arbitrary cube in $\mathcal{Q}$. Moreover,
$x_Q=2^{-\nu}k$ stands for the "lower left-corner" of $Q$. By $\ell(Q)=2^{-\nu}$ we indicate the side
length of $Q$, and by $|Q|=2^{-\nu n}$ its Lebesgue measure. 

We recall that a frame of wavelets consists of an admissible function
$\varphi$ subjected to translation and dilation, associated with an arbitrary dyadic cube $Q=Q_{\nu k}$,
\begin{equation}\label{phi7}
  \varphi_Q(x):=2^{n\nu/2} \varphi(2^{\nu}x-k)=|Q|^{1/2}\varphi_{\nu}(x-x_Q).
\end{equation}
For all dyadic $Q$ of length $\ell(Q)=2^{-\nu}$
\begin{equation}\label{phi5}
  \supp\hat\varphi_Q\subset \{\,\xi\mid K_0 2^{\nu}\le|\xi|\le K^02^{\nu}\,\} ,
\end{equation}
and since $(1+2^{\mu}|x-k|)^LD^\gamma\varphi(2^\nu x-k)$ is a bounded function, there are estimates
\begin{equation}\label{phi6}
   |D^\gamma \varphi_Q(x)|\leq C_{\gamma,L} |Q|^{-1/2-|\gamma|/n}(1+\ell(Q)^{-1}|x-x_Q|)^{-L}
\end{equation}
for each $L\in \mathbb{N}$ and multi-index $\gamma$ of length $|\gamma|\geq 0$. 

Moreover, we need pointwise estimates of convolutions with two parameters of
dilation and a translation; and it will often be crucial to have improved estimates in case one
factor has vanishing moments. So we recall that $\psi\in\cS$ is said to fulfill a moment
condition of order $M\ge-1$ if it annihilates all polynomials of degree $M$; that is, if
\begin{equation}
  \label{eq:moment}
   \int_{\bR^n} x^\alpha\psi(x)\,dx=0 \quad\text{for $|\alpha|\le M$}.
\end{equation}

\begin{lemma} \label{lem:p*p}
  If $\varphi,\psi\in\mathcal{S}(\bR^n)$ and $J$ is a dyadic cube of length $2^{-\mu}$ then there
  is for each $N>0$ a uniform estimate for $x\in\bR^n$ and $\nu\in\mathbb{Z}$,
  \begin{equation}
    \label{eq:p*p}
    (1+2^{\mu}|x-x_J|)^N|\psi(2^\mu(\cdot-x_J))*\varphi_\nu(x)|\le C_N2^{(N-n)(\mu-\nu)_+}.
  \end{equation}
When $\psi$ moreover fulfils a moment condition of order $M\in\bN_0$, then the above
improves to
  \begin{equation}
    \label{eq:p*pM}
    (1+2^{\mu}|x-x_J|)^N|\psi(2^\mu(\cdot-x_J))*\varphi_\nu(x)|\le
    C'_{N,M}2^{(N-n-(M+1))(\mu-\nu)_+}.
    \end{equation}
Similarly, when $\varphi$ satisfies a moment condition of order $M\in\bN_0$,
  \begin{equation}
    \label{eq:p*pM'}
    (1+2^{\mu}|x-x_J|)^N|\psi(2^\mu(\cdot-x_J))*\varphi_\nu(x)|\le C''_{N,M}2^{(N-n)(\mu-\nu)_+ -(M+1)(\nu-\mu)_+}.
  \end{equation}
In particular \eqref{eq:p*pM} or \eqref{eq:p*pM'} holds for all $M$ if $\psi$ or $\varphi$,
respectively, is admissible.
\end{lemma}
Details on these estimates can be found in Appendix~\ref{p*p-app}.

To elucidate the limitations of the present set-up, we first give the following account:

\begin{example}  \label{Meyer}
  In a fundamental contribution Lemari\'e and Meyer~\cite{MR864650} proved that an \emph{orthonormal} basis of Schwartz function
  wavelets exists in $L_2(\bR)$. This was briefly simplified by Meyer on p.~75
  in \cite{MR1228209} as resulting from an \emph{even} function $\theta_1\in C_0^\infty(\bR)$ with $\theta_1\ge0$ if
  \begin{equation} \label{Meyer0}
    \hat\psi(\xi)=\theta_1(\xi)e^{-\im\xi/2}  ,
  \end{equation}
  provided that $\theta_1(\xi)\ne0$ only holds for
  $2\pi/3\le|\xi|\le8\pi/3$, and that
  \begin{align} \label{Meyer1}
    \theta_1(\xi)^2+\theta_1(2\xi)^2&=1 \quad\text{for $2\pi/3\le |\xi|\le4\pi/3$},
\\  \label{Meyer2}
    \theta_1(\xi)^2+\theta_1(4\pi-\xi)^2&=1 \quad\text{for $4\pi/3\le |\xi|\le8\pi/3$}.    
  \end{align}
After a lengthy substantiation of the orthogonality, Daubechies described in her book
the result as relying on``quasi-miraculous cancellation''; cf.\ \cite[p.~119]{MR1162107}.
However, the cancellation simply comes from the orthogonality of even and odd functions on symmetric intervals $[-L,L]$:

Indeed, the possible $\theta_1$ are parametrised by the \emph{odd} $\chi\in C^\infty(\bR,\bR)$ with
$\chi\ge-1/2$ on $[0,\infty[\,$,
$\chi\equiv 1/2$ on $[\pi/3,\infty[\,$, so that $\chi\equiv -1/2$ on $\,]-\infty,-\pi/3]$, by means of 
\begin{equation} 
    \theta_1(\xi)^2  = \frac12+
    \begin{cases}
        \chi(\xi-\pi)   &\qquad\text{for $0<\xi\le 4\pi/3$},
    \\
        \chi(\pi-\xi/2)   &\qquad\text{for $\xi>4\pi/3$}.  
    \end{cases}
\end{equation}
Here \eqref{Meyer1} holds as $\chi$ is odd, whence we have
\eqref{phi4} if we take $\varphi=\psi$ and $\psi$ as in \eqref{Meyer0}. 

The induced wavelet family $2^{j/2}\psi(2^j x-k)$ is orthonormal: for two such wavelets $\psi'$,
$\psi''$ with $j'=j''$ we use $k=k'-k''$ and properties of parity and periodicity to get
\begin{equation}
  \begin{split}
  \langle \psi',\psi''\rangle &= 2\int_0^\infty \theta_1(\xi)^2 \cos(k\xi)\,d\xi/(2\pi)
\\
   &= \delta_{k,0}+\int_{-\pi/3}^{\pi/3} \chi(\eta)(\cos(k\eta+k\pi)+\cos(2k\eta))\,d\eta/\pi =\delta_{k',k''},    
  \end{split}
\end{equation}
since both cosines are even functions (the former is $-\cos(k\eta)$ for odd $k$).
If, say $j'=j''+ 1$ a substitution gives for $k=k'-2k''$, now because of the phase factor
$e^{-\im\xi/2}$ in \eqref{Meyer0},
\begin{equation}
  \begin{split}
  \langle \psi',\psi''\rangle & = 
     2\int_{2\pi/3}^{4 \pi/3} \sqrt2\theta_1(\xi)\theta_1(2\xi) \cos((k-\frac12)\xi)\,d\xi/(2\pi)
\\    &=\frac{\sqrt2}{\pi}\int_{-\pi/3}^{\pi/3}  (\frac14-\chi(\eta)^2)\sin((k-\frac12)\eta+k\pi)\,d\eta=0,      
  \end{split}
\end{equation}
as $\frac14-\chi^2$ is even and the sine is odd, also for odd $k$.  
The case $j'=j''-1$ is similar, and clearly $\langle \psi',\psi''\rangle = 0$ for $|j'-j''|\ge2$ by
the support condition on $\theta_1$.
(For the fact that the above orthonormal system spans a dense subspace, the
reader is referred to \cite{MR864650} or the literature on multiresolution analysis, e.g.\ the lucid exposition given
by Wojtasczyk~\cite[Sect.~3.2]{MR1436437} or by Frazier, Jawerth and Weiss~\cite[Thm.~7.11]{MR1107300}.)  
\end{example}

\begin{remark}
  \label{orto-rem}
Meyer's orthonormal wavelets basis for $L_2(\bR)$, which is recalled in Example~\ref{Meyer}, falls outside
the framework of Frazier and Jawerth~\cite{MR1070037}, and his $\psi$ is not even admissible in the
more general sense in Definition~\ref{adm}. In fact our set-up can never for $\psi=\varphi$ yield
an orthonormal basis of $L_2(\bR)$, for the normalisation would mean that
$2\pi =\int_{K_0\le|\xi|\le K^0} |\hat\varphi(\xi)|^2\,d\xi$,
and the constraint $K^0<\pi$ in \eqref{phi01} makes this impossible,
as $|\hat\varphi(\xi)|^2\le 1$ holds by \eqref{phi4}.
Frazier, Jawerth and Weiss~\cite{MR1107300} worked out a remedy by inserting the $\psi$ from
Example~\ref{Meyer} into $T_\psi$ and $\varphi=\psi$ into $S_\varphi$ and by a special argument obtained
$T_\psi S_\varphi=I$ on the isotropic space $\dot F^s_{p,q}$; cf.\ Theorem~7.20 in \cite{MR1107300}.
\end{remark}

\subsection{Maximal operators}
Let us recall some maximal inequalities pertaining to the Lebesgue space with mixed norm in \eqref{Lp}.

A fundamental tool is the maximal operator $M_k$ in the $x_k$-variable, $1\leq k\leq n$, for which we write $x=(x',x_k,x'')$,
whereby one of the groups $x'=(x_1,\dots,x_{k-1})$ and $x''=(x_{k+1},\dots,x_n)$ can be empty, to define for a locally integrable
function $f(x)$,
\begin{equation} \label{MK}
  M_k f(x)=\sup\limits_{I\in I_{x_k}} \dfrac{1}{|I|} \int_I |f(x',y_k,x'')| dy_k,
\end{equation}
whereby $I_{x_k}$ denotes the set of all intervals in $\bR$ containing $x_k$.

If $R$ is a rectangle $R=I_1\times\dots\times I_n$ it is easy to see from \eqref{MK} that 
\begin{equation}\label{M1}
  \int_R |f(x)| dx\leq |R|\cdot M_n(\cdots (M_1 f) \cdots)(x), \ \ \ \text{for every} \ x\in R.
\end{equation}
Usually we omit the parentheses in the repeated use of $M_1,\dots,M_n$.

For the mixed norms we use the following version of the Fefferman--Stein vector-valued maximal
inequality (cf.\ Stein~\cite{MR1232192}, and Bagby~\cite{MR0370171} for the mixed-norms):
if $\vec{p}=(p_1,\dots,p_n)$ for $0<p_1,\dots,p_n<\infty, \ 0<q\leq \infty$ and $0<t<\min(p_1,\dots,p_n,q)$ then
\begin{equation}\label{M2}
  \Big\|\Big(\sum_{\nu\in\bZ} \big(M_n\cdots M_1|f_{\nu}|^t\cdots)^{1/t} (\cdot)\big)^q \Big)^{1/q}\Big\|_{\vec{p}}
  \leq c\Big\|\Big(\sum_{\nu\in\bZ} |f_{\nu}|^q \Big)^{1/q}\Big\|_{\vec{p}}.
\end{equation}

In addition we need a well-known Peetre-type maximal inequality; cf.\ \cite{MR891189,MR2401510}:
for $t>0$ there exists a constant $c_t>0$, such that for every $f\in\cS'$
satisfying $\operatorname{supp}\hat{f}\subset [-2^{\nu},2^{\nu}]^{n}$ for some $\nu\in\bZ$, it holds for $x\in\bR^n$
and $\tau\ge n/t$ that
\begin{equation}\label{M3}
  \sup_{y\in\mathbb{R}^n} \frac{|f(y)|}{(1+2^{\nu}|y-x|)^{\tau}}\leq c_t \Big(M_n\cdots M_1 |f|^t\cdots\Big)^{1/t}(x).
\end{equation}

As a digression, we note the novelty that the spectral condition on $f$ is far from being
necessary, at least if $\tau>n/t$. In fact, as a corollary to Appendix~\ref{maximal_lemma}, cf.\ Remark~\ref{maximal-rem},
we have
\begin{proposition} \label{maximal-prop}
  When $\tau>n/t$, then Peetre's maximal inequality \eqref{M3} is also valid for 
  piecewise constant functions induced by a lattice of length $2^{\nu}$ in all variables for some $\nu\in \bZ$, that is,
  for functions having the form 
  $f(x)=\sum_{P\in\cQ,\ \ell(P)=2^{\nu}} a_P\mathds{1}_P(x)$ for $a_P\in\bC$.
\end{proposition}

\section{The $\varphi$-transform} \label{transform-sect}
As a general framework we use the space  $\cS'/\cP$ consisting of tempered
distributions modulo polynomials.
However, when considering $\cS'/\cP$ as a topological vector space we shall adopt the notation of
Triebel~\cite[Ch.~5]{MR781540} and write $\cS'/\cP$ as $\cZ'$, which is the dual space of
\begin{equation}
  \label{Z-id}
  \cZ(\bR^n)=\Big\{\,\psi\in\cS\mid \int_{\bR^n}x^\alpha\psi(x)\,dx = 0 \text{ for all multiindices $\alpha$}\,\Big\}.
\end{equation}
We recall that $\cZ$ is closed in $\cS$, hence a Fr\'echet space; and $\cZ'=\cS'/\cP$
as $\cP=\cZ^\perp$. 
Denoting the quotient map by $\Q\colon \cS'\to\cS'/\cP$, 
the $\operatorname{w}^*$-topology on $\cZ'$ is induced by the seminorms
$\Q f\mapsto |\langle f, \phi\rangle_{\cS'\times\cS}|$ parametrised (only) by $\phi\in\cZ$.
So trivially $\Q$ is a continuous linear operator:
\begin{equation}
  \label{eq:Qop}
  \Q\colon \cS'\to\cZ'.
\end{equation}
To recall the $\varphi$-transform $S_\varphi$ (a discrete wavelet transform),
we shall in the sequel adhere to the common practice of referring to any family $\{a_Q\}_{Q\in\cQ}$ as
a ``sequence'', even when the countable index set $\cQ$ is considered without any numbering.
Occasionally $(a)_Q$ will indicate the value of the sequence $a$ at $Q$. 

\begin{definition}\label{phi-transform}
Let $\varphi$ be an admissible function. The $\varphi$-transform $S_\varphi$ is the map sending
each $f\in\cS'/\cP$ to the complex-valued sequence
$S_\varphi f=\{(S_\varphi f)_Q\}_{Q\in\cQ}$ with 
\begin{equation*}
  (S_\varphi f)_Q= \langle f, \varphi_Q \rangle
 \quad\text{for all $Q\in \mathcal{Q}$}.
\end{equation*}
\end{definition}

When $\psi$ is admissible, $T_\psi$ is the linear operator  defined \emph{tentatively} on sequences
$a=\{a_Q\}_{Q\in\cQ}$ having finite support (i.e.\ $a_Q\ne 0$ only for finitely many $Q\in\mathcal{Q}$) by
\begin{equation}
  \label{eq:Tpsi}
  T_\psi a=\sum_{Q\in\cQ} a_Q\psi_Q.
\end{equation}
$T_\psi$ is the so-called inverse $\varphi$-transform when $\varphi$, $\psi$ are admissible and fulfil \eqref{phi4}.
We sometimes refer to $T_\psi$ as the \emph{synthesis} operator, and to  $S_\varphi$
as the analysis operator.

Furthermore, we need to make sense of $T_\psi a$ in a concise way for a variety of sequences $a$ without
finite support. In this case the summation in \eqref{eq:Tpsi} has no a priori meaning, despite
the countability of $\cQ$.
Indeed, the \emph{counting measure} $\tau_{1+n}$ on $\cQ\simeq \bZ\times \bZ^n$
does not suffice alone, since the sum in \eqref{eq:Tpsi} should be a distribution.

Now, each $\psi_Q$ can be identified with an element of $\cZ'$, for the quotient operator $\Q$ in
\eqref{eq:Qop} is injective on the subset of admissible functions, as e.g.\ $\hat\psi_Q(\xi)=0$ in a
neighbourhood of $\xi=0$. Therefore our aim is to make sense of $T_\psi a$ in $\cZ'$.

To sum the values of $Q\mapsto a_Q\psi_Q$ we take recourse to integration with respect to $\tau_{1+n}$ in a weak sense. 
More precisely, we shall use the notion of a Pettis integral (or weak Bochner integral) of a vector
function $f\colon X\to F'$ with respect to a measure $\mu$ on a $\sigma$-algebra $\bE$ in a set
$X$; and $F'$ being the dual of some Fr\'echet space $F$. 

Namely, such $f$ is said to be \emph{Pettis integrable } (or weakly integrable) if the scalar
function given on $X$ by $x\mapsto \langle f(x),v\rangle$ is in the 
Lebesgue space $L^1(\mu)$ for every vector $v\in F$ and, moreover, the dual space $F'$ contains some vector,
written $\int_X f\,d\mu$, such that
\begin{equation}
  \label{eq:Pettis}
  \left\langle \int_X f\,d\mu, v\right\rangle = \int_X \langle f, v\rangle \,d\mu \qquad \text{for all $v\in F$}.
\end{equation}
In general it is not easy to give sufficient conditions for the existence of the Pettis integral
(its uniqueness is obvious).
But as we show below, it is manageable in case of $T_\psi a$.

We consider $f(Q)= a_Q\psi_Q$ defined on $X=\cQ$, with $\mu=\tau_{1+n}$ and $F=\cZ$.
Then the basic criterion for Pettis integrability of $f$ is that 
$\int_{\cQ}|a_Q||\langle \psi_Q,\phi\rangle|\,d\tau_{1+n}<\infty$ for arbitrary $\phi\in \cZ$,
where $\langle\psi_Q,\phi\rangle$ stands for the action of $\psi_Q\in\cZ'$ on $\phi$.
We denote by $L_1(\mathcal{Q},\langle\psi_Q,\phi\rangle\,d\tau_{1+n})$ the space of such sequences,
and find the condition that $a\in L_1(\mathcal{Q},\langle\psi_Q,\phi\rangle\,d\tau_{1+n})$ for all $\phi\in \cZ$.

\begin{theorem} \label{Tpsi-thm}
When $\psi$ is admissible, then 
  the operator $a\mapsto T_\psi a$ in \eqref{eq:Tpsi} has an extension to a linear map
\begin{equation}
  \label{eq:1}
  \bigcap_{\phi\in \cZ} L_1(\mathcal{Q},\langle\psi_Q,\phi\rangle\,d\tau_{1+n})
  \xrightarrow{\ T_\psi\ }\cZ'(\bR^n),
\end{equation}
which on each sequence $a=\{a_Q\}_Q$ in the intersection is given as a distribution in $\cZ'(\bR^n)$ by the
formula, where $\phi\in \cZ(\bR^n)$,  
\begin{equation}
  \label{eq:4}
  \langle T_\psi a,\phi\rangle
  =\int_{\mathcal{Q}} a_Q\langle\psi_Q,\phi\rangle\,d\tau_{1+n}.  
\end{equation}
Moreover, $T_\psi a$ equals the $\operatorname{w}^*$-limit in $\cZ'(\bR^n)$ resulting
  from arbitrary approximation of $a=\{a_Q\}_Q$ by truncation to sequences having 
  finite, increasing and exhausting supports. This property determines the extension $T_\psi$ uniquely.
\end{theorem}
\begin{proof}
Obviously $\int_\mathcal{Q} a_Q\langle\psi_Q,\phi\rangle\,d\tau_{1+n}$
is well defined for every sequence $a$ belonging to the intersection in \eqref{eq:1} and every
$\phi\in \cZ$. 
For any numbering $Q_1, Q_2,\dots$ of the dyadic cubes, dominated convergence gives for $N\to\infty$,
\begin{equation}   \label{eq:2}
  \begin{split}  
  \left\langle\sum_{j=1}^{N} a_{Q_j}\psi_{Q_j},\phi\right\rangle 
  &=  \sum_{j=1}^{N} a_{Q_j}\langle\psi_{Q_j},\phi\rangle 
\\
  &=\int_{\cQ} \mathds{1}_{\{Q_1,\dots,Q_N\}} a_Q\langle\psi_Q,\phi\rangle\,d\tau_{1+n}
  \rightarrow \int_{\mathcal{Q}} a_Q\langle\psi_Q,\phi\rangle\,d\tau_{1+n}.
  \end{split}
\end{equation}
So according to the Banach--Steinhaus theorem for the Fr\'echet space $\cZ$, the
linear functionals $\varphi\mapsto \langle\sum_{j=1}^{N} a_{Q_j}\psi_{Q_j},\overline{\varphi}\rangle$ 
are equicontinuous $\cZ\to\bC$. In terms of increasing seminorms
$p_n(\phi)$ inducing the topology on $\cZ$, this means that for some $\delta>0$, $N\in\bN$ the
$0$-neighbourhood $\{\,\phi\mid p_N(\phi)<\delta\,\}$ is mapped into the unit ball in $\bR$;
so the functional $\Lambda$ defined as the above limit satisfies $|\Lambda(\phi)|\le
\delta^{-1}p_N(\phi)$ for $\phi\in\cZ$. Hence $\Lambda$ is continuous, i.e.\  
$\Lambda\in\cZ'(\mathbb{R}^n)$, and
\begin{equation}
  \label{eq:3}
  \langle \Lambda,\phi\rangle =   \lim_{N\to\infty}  \langle \sum_{j=1}^{N} a_{Q_j}\psi_{Q_j},\phi\rangle 
  =\int_{\mathcal{Q}} a_Q\langle\psi_Q,\phi\rangle\,d\tau.
\end{equation}
As the right hand side is independent of the numbering, so is $\Lambda$, i.e.\ $\Lambda$ depends
only on $a$ and $\psi$. Setting $T_\psi a=\Lambda$ we obtain \eqref{eq:4}. 

Clearly $a\mapsto T_\psi a$ is a linear map, and if $a$ has finite support a little algebra as in \eqref{eq:2} shows
that the finite sum $\sum_Q a_Q\psi_Q$ equals $T_\psi a$; so $T_\psi$ extends the map \eqref{eq:Tpsi}. 

Now let a sequence $a=\{a_Q\}_{Q\in\cQ}$, given in the intersection in \eqref{eq:1}, be
approximated by sequences $a^{(m)}$ of
finite support by truncation. As $a$ is a function on $\cQ$, this may be written in terms of characteristic functions of finite
sets $\cQ_m\subset\cQ$ as 
\begin{equation}
  a^{(m)}=\mathds{1}_{\cQ_{m}}a, \quad\text{for $m\in\bN$}.
\end{equation}
If $\cQ_1\subset \cQ_{2}\subset\dots$ and $\bigcup_m \cQ_m= \bR^n$, the supports of the
$a^{(m)}$ are increasing and they exhaust $\supp a$. Then it follows analogously to 
 \eqref{eq:2} that $T_\psi a^{(m)}= \sum_{Q\in\cQ_m} a_{Q}\psi_{Q}$ converges in
 $\cZ'(\bR^n)$ to the distribution $T_\psi a$; cf.~\eqref{eq:4}.

Finally, whenever $\tilde T_\psi$ extends the map in
\eqref{eq:Tpsi} to an operator as in \eqref{eq:1} having the property
just obtained for $T_\psi$, for some approximation $a^{(m)}$ with the properties above, then 
$\tilde T_\psi a=\lim_m \tilde T_\psi a^{(m)} =\lim_m T_\psi a^{(m)} = T_\psi a$ in $\cZ'$.
\end{proof}

Notice that $T_\psi a = \sum_Q a_Q\psi_Q$ whenever $\{a_Q\}$ has finite support. In general the sum 
must be understood as the Pettis integral $T_\psi a=\int_{\cQ}a_Q\psi_Q\,d\tau_{1+n}$ given by
\eqref{eq:4}; although the latter looks less intuitive. 
We proceed to show that it has a number of desired properties of $T_\psi$.

In practice the convergence questions related to application of $T_\psi$ may often be handled via the
classical estimate for $N>n$,
\begin{equation}  \label{kN-est}
  \sum_{k\in\bZ^n}(1+|k|)^{-N}\le \sum_{k_1,\dots,k_n\in\bZ}\prod_{j=1}^n (1+|k_j|)^{-\frac{N}n}
  \le \Big(1+2\sum_{m=1}^\infty m^{-\frac{N}n}\Big)^n <\infty.
\end{equation}
This is useful e.g.\ for the basic result that $T_\psi$ always is defined on a sequence of wavelet
coefficients:

\begin{proposition}
  \label{TS-prop}
 When a sequence $a=S_\varphi f$ for some $f\in\cS'/\cP$, then the general synthesis operator
 $T_\psi$ in Theorem~\ref{Tpsi-thm}, with any admissible $\psi$, is defined on $a$ and 
 \begin{equation}
   \label{eq:TS}
   \langle T_\psi(S_\varphi f),\phi\rangle =
  \sum_{\nu=-\infty}^\infty \sum_{k\in\bZ^n} (S_\varphi f)_{Q_{\nu k}} \langle\psi_{Q_{\nu k}} ,\phi\rangle
  \qquad \text{for $\phi\in\cZ$},
\end{equation}
where the terms are $\tau_{1+n}$-integrable on $\cQ$.
In short, $R(S_\varphi)\subset D(T_\psi)$.
\end{proposition}
\begin{proof}
  To verify the $L_1$-condition in \eqref{eq:1} we shall prove 
$\sum_Q|\langle f,\varphi_Q\rangle||\langle \psi_Q,\phi\rangle|$ 
finite for every $\phi\in\cZ$. Since $f$ is temperate we have for
some $d>0$, if $Q=Q_{\mu k}$,
\begin{equation}
  \begin{split}
  |\langle f,\varphi_Q\rangle|
   &\le c\sum_{|\alpha|\le d}\sup_x|(1+|x|)^d D^\alpha(2^{n\mu/2}\varphi(2^\mu x-k)|
\\
   &\le c\sum_{|\alpha|\le d}\sup_y|(1+2^{-\mu}|y+k|)^d 2^{\mu(|\alpha|+n/2)}D^\alpha\varphi(y)|
\\
   &\le c(\varphi,d)2^{d|\mu|+\mu n/2}(1+|k|)^d.
  \end{split}
\end{equation}
To invoke Lemma~\ref{lem:p*p}, we observe that $\langle \psi_Q,\phi\rangle
=\psi_Q*\tilde\phi(0)$, so the estimates there apply for
$\nu=0$ and $x=0$. Thus $(1+2^\mu|x-x_Q|)^N= (1+|k|)^N$ for our cube $Q$.
We may therefore apply \eqref{kN-est} for $N=n+1$ if we take $N=d+n+1$ in Lemma~\ref{lem:p*p}. 

Indeed, we can make a crude estimate thus: for $\mu\ge 0$ we may use \eqref{eq:p*pM} for any $M>0$ as $\psi$ is
admissible, or for $\mu<0$ note that $\phi$ as a member of $\cZ$
fulfils \eqref{eq:p*pM'} for every $M>0$, to get 
\begin{equation}
  |\langle \psi_Q,\phi\rangle|\le C 2^{-|\mu|(M-d-1-n/2)}(1+|k|)^{-d-n-1}.
\end{equation}
So using Tonelli's theorem to pick a summation order for $Q=Q_{\mu k}$, 
and by taking e.g.\ $M>2d+1+n$, we have $\sum_{\mu\in\bZ}\sum_{k\in\bZ^n}|\langle f,\varphi_Q\rangle||\langle
\psi_Q,\phi\rangle|<\infty$. By Fubini's theorem the integral in \eqref{eq:4} therefore equals
the iterated sum in \eqref{eq:TS}.
\end{proof}

As a virtue of the proposition, it makes sense to study $T_\psi(S_\varphi f)$ and to derive
from \eqref{phi4} the wavelet decomposition formula
\begin{equation} \label{wave-id}
  T_\psi(S_\varphi f) = f \quad\text{ for all $f\in\cS'/\cP$}.  
\end{equation}
In two steps we give a proof based on our explicit definition of $T_\psi$ in Theorem~\ref{Tpsi-thm}.

The first step is to subject the $f$  in \eqref{wave-id} to Peetre's homogeneous Littlewood--Paley
decomposition, which we have derived in a precise version in Appendix~\ref{LP'-app}:
when condition \eqref{phi4} is satisfied by $\varphi$, $\psi$,
then $\hat\phi:=\cF\psi\overline{\cF\varphi}=\cF\psi\cF\tilde\varphi$ obviously fulfils \eqref{eq:xi0},
so by Proposition~\ref{LP-prop} there is, for each $f\in\cS'$, specific polynomials $P_{m,N}$ of degree
$m$ fulfilling
 \begin{equation} \label{eq:LWPP'}
    f= \sum_{\nu=-N}^\infty \psi_\nu*\tilde \varphi_\nu*f+P_{m,N}+R_m \qquad\text{in}\quad \cS'.
 \end{equation}
Here the remainder fulfils $\langle R_m,\omega\rangle=\cO(2^{-N(n+m+1-d)})$ for every
$\omega\in\cS$, with $d\ge0$ denoting the $\cS'$-order of $f$; cf.\ \eqref{S'-order} for the notion.
For $d\le n$ this estimate yields exponentially
fast convergence to $f$ in $\cS'$ for $N\to\infty$; the full statement in
Proposition~\ref{LP-prop} moreover shows that $P_{m,N}$'s individual terms $c_{\alpha,N}x^\alpha$ are $\cO(2^{-N(n+|\alpha|-d)})$ 
in $\cS'$-seminorm, so that in case $d<n$  even $P_{m,N}=\cO(2^{-N})\to0$  for $N\to\infty$.

The polynomials $P_{m,N}$ are asymptotically uniquely given (as the Taylor polynomials of a 
convolution $2^{-nN}\Phi(2^{-N}\cdot)*f$, cf.~\eqref{fPhi-id}), but since the degree $m$
is at our disposal, we can for a general temperate order $d$ of $f$ arrange that $n+m+1>d$, so that
at least the remainder $R_m$ converges to $0$ in $\cS'$ for $N\to\infty$. However,
the terms of $P_{m,N}$ with $|\alpha|\le d-n$ do not necessarily go to $0$, so the convergence of
$\sum_{\nu\ge-N}\psi_\nu*\tilde \varphi_\nu*f$ to $f$ for $N\to\infty$ is only obtained in $\cZ'=\cS'/\cP$.

To convert \eqref{eq:LWPP'} into a summation over $\cQ$, we need a convenient result asserting that certain
convolution integrals over $\bR^n$ can be replaced by convolution over a discrete subgroup. This
fact is not surprising, and we state the result in Lemma~\ref{conv-prop} below providing  both a
pointwise limit and convergence in $\cS'$ of the sum \eqref{conv-id}. 
The result of the lemma was used implicitly in \cite{MR2186983,MR808825}.

\begin{lemma}
  \label{conv-prop}
If $\phi\in\cS$ and $g\in\cS'$ satisfy that $\supp\hat\phi$ and $\supp\hat g$ are subsets of $\,]-L,L[\,^n$ for
some $L>0$, then
\begin{equation}  \label{conv-id}
  \phi* g(x)=\left(\frac\pi{L}\right)^n \sum_{k\in\bZ^n} \phi\left(x-\frac\pi{L}k\right)g\left(\frac\pi{L}k\right).
\end{equation}
The sum converges absolutely and unconditionally in the Fr\'echet space $C^\infty(\bR^n)$ 
and in $\cS'(\bR^n)$.
\end{lemma}
The optimality of the constant $\pi/L$ was amply elucidated by Meyer~\cite[Thm.1.1]{MR1228209}.

\begin{proof}
  If we apply Poisson's summation formula to the Schwartz function $\phi(x-\cdot) g$,
  cf.~\cite[(7.2.1)']{MR1996773},  we have for any $a>0$ (with pointwise convergence for $x\in\bR^n$)
  \begin{equation}
    \big(\frac{2\pi}{a}\big)^n\sum_{k\in\bZ^n} \phi\big(x-\frac{2\pi}{a}k\big)
    g\big(\frac{2\pi}{a} k\big) = \sum_{k\in\bZ^n} \cF(\phi(x-\cdot)g)(ak).
  \end{equation}
Here the Fourier transformed product is supported in $]-2L,2L[\,^n$ by the support rule of
convolutions. So for $a=2L$ the sum on the right-hand side is trivial for $k\ne 0$; i.e.\ the
sum equals $\cF(\phi(x-\cdot)g)(0)=\int_{\bR^n}\phi(x-y)g(y)\,dy= \phi*g(x)$.

For the absolute convergence we shall prove 
$\sum_{k\in\bZ^n} \sup_{x\in K} |D^\alpha \phi(x-\frac\pi{L}k)g(\frac\pi{L}k)|$ finite for
arbitrary multiindices $\alpha$ and compact sets $K\subset\bR^n$. Since $g$ is a slowly increasing
function we have for some $N>0$,
\begin{equation}  \label{g-est}
  \left|g\big(\frac\pi{L}k\big)\right| \le c\frac{(1+|x-\frac\pi{L}k|)^{N+n+1}(1+|x|)^{N+n+1}}{(1+|\frac\pi{L}k|)^{n+1}}.
\end{equation}
Here $(1+|x|)^{N+n+1}$ is bounded on the set $K$, and so is $(1+|x-\frac\pi{L}k|)^{N+n+1}$ times the Schwartz
function $D^\alpha\phi(x-\frac\pi{L}k)$ on $K\times\bZ^n$, so finiteness results at once from
\eqref{kN-est}.

In view of this, any numbering of the $k\in\bZ^n$ induces a Cauchy sequence in $C^\infty(\bR^n)$,
where the limit must equal $\phi*g$ by the first part of the proof. Thus it is independent
of the numbering.

For $\cS'$ the statement also boils down to \eqref{kN-est}, for  the finiteness
of $\sum_{k\in\bZ^n} |\langle \phi(x-\frac\pi{L}k)g(\frac\pi{L}k),\psi\rangle |$ follows when $\psi\in\cS$ from a
uniform bound, which via \eqref{g-est} is
reduced to a test of the fixed function $(1+|x|)^{N+n+1}$ against
$2^p\psi(x)\phi(x-\frac\pi{L}k)(1+|x-\frac\pi{L}k|^2)^{p}$, $p=(N+n+1)/2$, which by virtue of $\psi$ runs through a bounded set
in $\cS$ as $k$ varies in $\bZ^n$. Now the limit theorem for $\cS'$ yields (unconditional)
convergence to some $u\in\cS'$, which in $\cD'$ coincides with the limit in $C^\infty$, so $u=\phi*g$.
\end{proof}

As the second step towards \eqref{wave-id} 
we may apply Lemma~\ref{conv-prop} to $g=\tilde\varphi_\nu* f$ and $\phi=\psi_\nu$, taking
$L=2^\nu\pi$ to get a clean formulation; cf.\ \eqref{phi5} and the constraint $K^0<\pi$ in
\eqref{phi01}. Then \eqref{eq:LWPP'} gives
 \begin{equation} \label{eq:LWPP''}
    f= \sum_{\nu=-N}^\infty \big(2^{-n\nu}\sum_{k\in\bZ^n}\psi_\nu(x-2^{-\nu}k)
    \tilde\varphi_\nu*f(2^{-\nu}k)\big)+P_{m,N}+R_m.
 \end{equation}
Here $2^{-n\nu}=(|Q|^{1/2})^2$ is a product of two normalisation factors, one of which yields
the factor $\psi_Q(x)$ in the sum, cf.\ \eqref{phi7}, and since we use sesqui-linear pairings,
\begin{align} \label{eq:psiQ}
   2^{-n\nu/2}\psi_\nu(x-2^{-\nu}k)&=\psi_Q(x)
\\
  \label{eq:phiQ}
   2^{-n\nu/2}\tilde\varphi_\nu*f(2^{-\nu}k)&=\langle f,
|Q|^{1/2}\varphi_\nu(\cdot-2^{-\nu}k)\rangle= \langle f, \varphi_Q\rangle.
\end{align}
The latter expression equals $S_\varphi\Q f$, but we usually just write $S_\varphi f$, as
any polynomial clearly can be added to $f$ without changing its wavelet coefficients.

We now obtain the decomposition of each $f\in\cZ'$ as a
``linear combination'' of the wavelets $\psi_Q(x)$. More precisely, it is a Pettis integral of the
building blocks $\langle f,\varphi_Q\rangle \psi_Q (x)$,
although for convenience we simply write it as a sum.

\begin{proposition}   \label{l:identity}
When $\varphi,\psi$ are admissible and satisfy \eqref{phi4}, then one has for every
$f\in\cS'/\cP$ that, with convergence in $\cZ'(\bR^n)$,
\begin{equation}  \label{TS-id}
  f(x)=\sum\limits _Q \langle f,\varphi_Q\rangle \psi_Q (x)=T_\psi \circ S_\varphi f(x).
\end{equation}
More precisely, there is $\cS'$-convergence to $f$ of the iterated sum in
\eqref{eq:LWPP''} with the identifications \eqref{eq:psiQ}--\eqref{eq:phiQ}; and error term
$R_m=\cO(2^{-N(n+m+1-d)})$, cf.\ \eqref{eq:LWPP'}.
\end{proposition}
\begin{proof}
  In \eqref{eq:LWPP''} we take $m$ so large that $n+m+1>d$, whence $R_m\to0$ for $N\to\infty$ as discussed 
  after \eqref{eq:LWPP'}.
  Then the outer sum in \eqref{eq:LWPP''} converges in $\cS'$ for $N\to\infty$, cf.\ \eqref{eq:LWPP'} or
  the appendix. And the inner sum converges in $\cS'$ to a temperate distribution, according to the last
  statement of Lemma~\ref{conv-prop}, so the quotient operator $\Q$ 
  commutes with both summations for $f\in\cS'$ because of its continuity in \eqref{eq:Qop}.
  So \eqref{eq:phiQ} gives
  \begin{equation} \label{Qf-id}
    \Q f = \sum_{\nu=-\infty}^\infty \Q(\sum_{k\in\bZ^n} \langle f,\varphi_Q \rangle \psi_Q)
         = \sum_{\nu=-\infty}^\infty \sum_{k\in\bZ^n} (S_\varphi f)_Q \Q\psi_Q
    \qquad\text{in $\cZ'$}.
  \end{equation}
Thus we get the first formula in \eqref{TS-id} for the elements of $\cS'/\cP$ (as $\Q\psi_Q$
identifies with $\psi_Q$), if we just write a sum with respect to $Q$.
Applying the continuous functional $\langle\cdot,\phi\rangle$ on both sides of \eqref{Qf-id} for an arbitrary
$\phi\in\cZ$, 
Proposition~\ref{TS-prop} gives that $\langle \Q f,\phi\rangle=\langle 
T_\psi(S_\varphi f),\phi\rangle$; whence the second formula in \eqref{TS-id}. The final remarks on
$\cS'$-convergence and $R_m$ was seen prior to the statement.
\end{proof}

As a corollary to the proof, note that when $\langle\cdot,g\rangle$ is applied to both sides of
\eqref{Qf-id}, since 
$\langle \psi_Q,g\rangle=\int \psi_Q\overline{g}\,dx=\overline{(S_\psi g)_Q}$,
one obtains for all $f\in\cZ'$, $g\in\cZ$,
  \begin{equation} \label{dual-id}
    \langle f,g\rangle   = \sum_{\nu=-\infty}^\infty \sum_{k\in\bZ^n} (S_\varphi f)_Q
    \overline{(S_\psi g)_Q}.
  \end{equation}
More intuitively, one could obtain this by noting that both sides of \eqref{dual-id} equal the following,
whenever $\cQ_1\subset \cQ_2\subset\dots$ are finite subsets fulfilling $\cQ=\bigcup_m\cQ_m$,
\begin{equation}
  \lim_{m\to\infty} \sum_{Q\in\cQ_m} \langle (S_\varphi f)_Q\psi_Q, g\rangle.
\end{equation}
This follows on the left of \eqref{dual-id} from insertion of $f=T_\psi(S_\varphi f)$ and use of the
$\operatorname{w}^*$-approximation property in  
Theorem~\ref{Tpsi-thm}; on the right from the $\tau_{1+n}$-integrability
in Proposition~\ref{TS-prop} and dominated convergence.

To simplify \eqref{dual-id} one may invoke the scalar product $\langle s,t \rangle=\sum_{Q\in\cQ} s_Q\overline{t_Q}$
defined for those sequences $s=\{s_Q\}_{Q\in\cQ}$ and $t=\{t_Q\}_{Q\in\cQ}$ for which the product $\{s_Q\overline{t_Q}\}_{Q\in\cQ}$ is in
$\ell_1(\cQ,\tau_{1+n})$: by the abovementioned integrability, \eqref{phi4} now more simply implies
  \begin{equation} \label{dual-id'}
    \langle f, g\rangle   = \langle S_\varphi f, S_\psi g\rangle
   \qquad\text{for all $f\in\cZ'$, $g\in\cZ$}. 
  \end{equation}
This justifies the ``Parseval identity'' alluded to by Frazier and Jawerth~\cite{MR1070037}, and 
these formulas can now be further specialised as done in \cite{MR808825,MR1070037}.

It is also obvious that $P=S_\varphi\circ T_\psi$ is a projection, or more precisely an
idempotent when \eqref{phi4} holds, as $T_\psi \circ S_\varphi=I$ then; cf.\
Proposition~\ref{l:identity}. In fact, $P=I$ holds 
precisely when the wavelet sequences $\{\varphi_Q\}_{Q\in\cQ}$ and $\{\psi_Q\}_{Q\in\cQ}$ form a biorthogonal system, i.e.\
\begin{equation} \label{biortho-id}
  \langle \psi_Q,\varphi_J\rangle = \int_{\bR^n}\psi_Q(x)\overline{\varphi_J(x)}\,dx =\delta_{Q,J}
  \qquad\text{(Kronecker delta)}.
\end{equation}
Although this is known in other contexts, it is perhaps instructive to note how nicely a formal
proof fits with the definition of $T_\psi$ by the Pettis integral:

\begin{proposition}
  \label{ST-prop}
If \eqref{phi4} holds, the identity $S_\varphi\circ T_\psi=I$ is equivalent to \eqref{biortho-id}
\end{proposition}
\begin{proof}
  By definition of $S_\varphi$ and $T_\psi$, any $a$ in $D(T_\psi)$ satisfies
  \begin{equation}
    \label{ST-id}
    (S_\varphi(T_\psi a))_J= \langle T_\psi a, \varphi_J\rangle = \int_{\cQ} a_Q\langle
    \psi_Q,\varphi_J\rangle\,d\tau_{1+n}
    \qquad \text{for all $J\in\cQ$}.
  \end{equation}
  When $S_\varphi\circ T_\psi=I$ holds, then insertion of the sequence $a=\{\delta_{Q,J_0}\}_{Q\in\cQ}$ from $D(T_\psi)$,
  cf.\ \eqref{eq:1}, shows that 
  $\delta_{J,J_0}=\int \delta_{Q,J_0}\langle \psi_Q,\varphi_J\rangle\,d\tau_{1+n} =
  \langle\psi_{J_0},\varphi_J\rangle_{\cS',\cS}$.
  The converse is clear from \eqref{ST-id}.
\end{proof}

In the present set-up, biorthogonality is also equivalent to the property that the wavelets
give rise to a basis of $\cZ'$; and this is always unconditional:

\begin{theorem} \label{basis-thm}
  If $\varphi$, $\psi\in\cS$ are admissible and fulfil the reconstruction identity
  \eqref{phi4}, then the following properties are equivalent:
  \begin{itemize}
  \item[{\rm(i)}] $S_\varphi\circ T_\psi=I$;
  \item[{\rm(ii)}] the biorthogonality condition \eqref{biortho-id} holds;
  \item[{\rm(iii)}] there exists a numbering $Q_1, Q_2,\dots$ of the dyadic cubes for which
    the corresponding sequence $\psi_{Q_j}(x)=2^{n\nu_j/2}\psi(2^{\nu_j} x-k_j)$, $j\in\bN$, constitutes a basis of
    $\cZ'$;
 \item[{\rm(iv)}] every numbering of $\cQ$ induces
    a basis $\{\psi_{Q_j}\}_{j\in\bN}$ of $\cZ'$.
  \end{itemize}
In the affirmative case every numbering $\{\psi_{Q_j}\}_{j\in\bN}$ is an \emph{unconditional} basis
of $\cZ'$.
\end{theorem}

\begin{remark} \label{numbering-rem}
We observe that \emph{any} numbering $Q_1, Q_2,\dots$ of $\cQ$
yields $f=\sum_{j=1}^\infty a_{Q_j}\psi_{Q_j}$ in $\cZ'$ for $a=S_\varphi f$. This results
from the $\operatorname{w}^*$-approximation property in
  Theorem~\ref{Tpsi-thm} (cf.\ its proof for the notation) by taking $\cQ_m=\{Q_1,\dots,Q_m\}$, 
  which yields for arbitrary $g\in\cZ$,   
  \begin{equation} \label{fsum}
    \langle f,g\rangle=\langle T_\psi a,g\rangle= \lim_{m\to\infty} \langle \sum_{j=1}^m a_{Q_j}\psi_{Q_j}, g\rangle=
     \langle \sum_{j=1}^\infty a_{Q_j}\psi_{Q_j}, g\rangle .
  \end{equation}
  More precisely, by the Banach--Steinhaus theorem there is a distribution
  $\sum_{j=1}^\infty a_{Q_j}\psi_{Q_j}$ in $\cZ'$ satisfying the last identity for $g\in\cZ$ (cf.\
  \eqref{eq:2} ff.), 
  whence $f$ is the sum of the series.
\end{remark}

\begin{proof} That (i) and (ii) are equivalent is just a restatement of Proposition~\ref{ST-prop}.

  Clearly (iv) implies (iii). Given (iii), there is to each $f\in\cZ'$ unique
  scalars $c_j$ such that $f=\sum_{j=1}^\infty c_j\psi_{Q_j}$. 
  Now also $f=\sum_{j=1}^\infty a_{Q_j}\psi_{Q_j}$ holds in $\cZ'$, cf.\ \eqref{fsum}, so
  $c_j=a_{Q_j}=\langle f,\varphi_{Q_j}\rangle$. For $f=\psi_Q$ it is obvious that
  $c_j=\delta_{Q_j,Q}$ whilst $\langle f,\varphi_{Q_j}\rangle=\langle \psi_Q,\varphi_{Q_j}\rangle$,
  whence (ii) holds.
 
  When (ii) is satisfied and a numbering is given, then in addition to the existence in
  \eqref{fsum} we get uniqueness thus: if also
  $\sum_{j=1}^m b_{Q_j}\psi_{Q_j}$ converges to $f=T_\psi a$, using (i) and (ii) and applying  
  the continuous functionals $\langle\cdot,\varphi_{Q_k}\rangle$ on $\cZ'$,
  \begin{equation}
    b_{Q_k}=\langle T_\psi a,\varphi_{Q_k}\rangle=(S_\varphi(T_\psi a))_{Q_k}=a_{Q_k} \quad\text{for any $k$}.
  \end{equation}
  Moreover $f\mapsto\langle f,\varphi_{Q_j}\rangle=a_{Q_j}$ is continuous, so 
  $\{\psi_{Q_j}\}$ is a Schauder basis for $\cZ'$; even an
  unconditional basis, as any bijection $p\colon\bN\to\bN$ via
  $\cQ_m=\{Q_{p(1)},\dots,Q_{p(m)}\}$ in the $\operatorname{w}^*$-approximation property
  gives $\sum_{j=1}^\infty a_{Q_{p(j)}}\psi_{Q_{p(j)}}=T_\psi a=f$. Hence (iv) and last part are obtained.
\end{proof}

Besides being an ingredient in the proof, our Remark~\ref{numbering-rem} also has an important
consequence. Namely, it shows that $T_\psi a$ can be calculated by summing the terms $a_Q\psi_Q$ in
any order, which was far from obvious at the outset. 

More surprisingly, this possibility of rearrangement is valid regardless of whether the wavelets $\psi_Q$ via any
numbering give an unconditional basis of $\cZ'$ or not (as the remark was made prior to the proof of Theorem~\ref{basis-thm}).
But it all comes out naturally by using the definition of $T_\psi$ by the Pettis integral, 
ultimately because this gave the precise description of its domain by the integrability condition
in \eqref{eq:1}. Summing up we have

\begin{corollary} \label{Tpsi-cor}
  For the synthesis operator $T_\psi$ in \eqref{eq:1}--\eqref{eq:4} the value on every sequence $a=\{a_Q\}$ in
  $D(T_\psi)$ is obtained from $T_\psi a=\sum_{j=1}^\infty a_{Q_j}\psi_{Q_j}$ by summing the terms
  as an infinite series in $\cZ'$ in any order.
\end{corollary}

\begin{remark} \label{Meyer-rem}
  In view of Theorem~\ref{basis-thm} it would be interesting to know whether or not there exist
  wavelets in our framework that fulfil the biorthogonality condition \eqref{biortho-id}. We
  envisage that an explicit construction would require more than a single generator for $n>1$, but to keep the presentation
  simple, we have left this aspect to the future.
\end{remark}

\begin{remark}
  The explanation of Frazier and Jawerth~\cite{MR808825,MR1070037} left quite a
  burden with the reader; e.g.\ no argument was given for the
  (heuristically obvious) identification of the sum over $Q$ in \eqref{TS-id} 
  with $T_\psi\circ  S_\varphi$. 
  They did account for pointwise convergence  in \eqref{conv-id}, 
  but the $\cS'$-convergence in Proposition~\ref{conv-prop}
  was first stated by Bownik and Ho~\cite{MR2186983}.
  As the formula $T_\psi(S_\varphi f)=f$ only makes sense if $R(S_\varphi)\subset
  D(T_\psi)$, let us also point out that Proposition~\ref{TS-prop} seems to be a novelty---and
  so is the unified result in Theorem~\ref{basis-thm} that wavelets form unconditional bases in the
  topological vector space $\cZ'$ if and only if they are biorthogonal.
\end{remark}

\section{Mixed-norm Triebel--Lizorkin spaces} \label{MTL-sect}
We are now ready to introduce the homogeneous Triebel--Lizorkin spaces $\dot{F}^s_{\vec{p},q}$ based on the mixed norms 
\eqref{Lp}. These spaces generalize the classical
homogeneous Triebel--Lizorkin spaces (cf.\ \cite{MR781540,MR1163193} or \cite{MR1070037}), which
can easily be recovered from the special choice $\vec{p}=(p,\dots,p)$, for $0<p<\infty$, in the following definition.

\begin{definition}\label{def:T-L}
For $s\in\mathbb{R},\;\vec{p}=(p_1,\dots,p_n),$ with $0<p_1,\dots,p_n<\infty,\;0<q\leq \infty$ and
$\varphi$ admissible in the sense of Definition~\ref{adm}, the homogeneous mixed-norm Triebel--Lizorkin space
$\dot{F}^s_{\vec{p},q}$ is the set of all $f\in \cS'/\cP$ such that  
\begin{equation}\label{TLnorm}
  \|f\|_{\dot{F}^s_{\vec{p},q}}:=\Big \| \Big( \sum _{\nu\in \mathbb{Z}} (2^{\nu s}|\varphi_{\nu}\ast f|)^q \Big) ^{1/q}\Big \|_{\vec{p}}<\infty,
\end{equation}
with the $l_q$-norm replaced by $\sup _{\nu}$ for $q=\infty$.
\end{definition}

For a classical general reference on Sobolev spaces with mixed norms the reader is referred to the 
books by Besov, Il'in and Nikol'ski{\u\i}~\cite{MR519341,MR1450401}, the latter with extension to Triebel--Lizorkin spaces.
A small account of the isotropic spaces $\dot F^s_{p,q}$ can be found in the work of Runst and Sickel~\cite{MR1419319}.

At this point the reader should just consider $\dot{F}^{s}_{\vec{p},q}$ as a quasi-normed space. 
In the end of the section we shall obtain $\varphi$-independence of the above space and equivalent
quasi-norms, subject to \eqref{phi01}, together with its completeness. 

For now we mention the basic embeddings, where $\hookrightarrow$ is understood to mean linear continuous
injection. 

\begin{proposition} \label{ZZ'-prop}
For $s\in\bR$, $p_1,\dots,p_n\in\,]0,\infty[\,$ and $q\in\,]0,\infty]$, while $u\in[q,\infty]$, 
 \begin{gather} \label{qu-emb}
  \dot{F}^s_{\vec{p},q}\hookrightarrow \dot{F}^s_{\vec{p},u},
\\ 
  \label{eq:ZZ'}
  \cZ\hookrightarrow \dot{F}^s_{\vec{p},q} \hookrightarrow \cZ'.
\end{gather}
\end{proposition}
\begin{proof}
The first line is trivial since $\ell_q\hookrightarrow\ell_u$.  
Continuity of $\cZ\hookrightarrow \dot{F}^s_{\vec{p},q}$ 
follows from Lemma~\ref{lem:p*p} by taking $\mu=0$ and
$x_J=0$ there: if $\psi\in\cZ$ then $(1+|x|)^{-N}$ is in $L_{\vec p}$ for some $N>n$ 
(say $p_jN/n>1$, all $j$); and the number
of vanishing moments $M$ in \eqref{eq:p*pM} can be taken so large that $-s+N-n-(M+1)<0$, which applies for
$\nu<0$ in \eqref{TLnorm}; whilst for $\nu\ge 0$ one can invoke \eqref{eq:p*pM'} for $M$ so large
that $s-(M+1)<0$, as $\varphi$ is admissible. This way $\|\psi\|_{\dot F^s_{\vec p,q}}\le
c(C'_{N,M}+C''_{N,M})$, where both $\psi\mapsto C'_{N,M}$ and $\psi\mapsto C''_{N,M}$ are continuous seminorms on
$\cS$; cf.\ the proof of Lemma~\ref{lem:p*p}.

Now $\dot{F}^s_{\vec{p},q} \hookrightarrow \cZ'$ follows from Proposition~\ref{LP-prop} by setting
$\hat\phi=\overline{\hat\psi}\hat\varphi$ with $\psi$ as in \eqref{phi4}, for if $f\in \dot F^s_{\vec p,q}$ and $g\in\cZ$,
\begin{equation}
   |\langle f,g\rangle|=|\langle \sum_{\nu\in\bZ} \tilde\psi_\nu*\varphi_\nu*f,g\rangle| 
   \le\sum_{\nu\in\bZ}\|\varphi_\nu*f\|_\infty\|\psi_\nu*g\|_1.
\end{equation}
Here $\|\varphi_\nu*f\|_\infty\le c2^{\nu(\frac1{p_1}+\dots+\frac1{p_n})}\|\varphi_\nu* f\|_{\vec p}$,
since $\varphi_\nu*f$ has spectral radius $K^02^{\nu}$;
cf.\ \cite[Prop.~4]{MR2319603}. By adding more terms in the $L_{\vec p}$-norm,
\begin{equation}
   |\langle f,g\rangle|
   \le c\|f\|_{\dot F^s_{\vec p,q}}\sum_{\nu\in\bZ}2^{\nu(\frac1{p_1}+\dots+\frac1{p_n}-s)}\|\psi_\nu*g\|_1.
\end{equation}
The final sum is finite, in fact bounded from above by a Schwartz seminorm on $g$ as one can see by
adapting the above proof of $\cZ\hookrightarrow \dot{F}^s_{\vec{p},q}$.
Hence $|\langle f,g\rangle| \le c\|f\|_{\dot F^s_{\vec p,q}}$.
\end{proof}

\begin{proposition}
There are Sobolev embeddings in the $\dot F^s_{\vec p,q}$-scale, namely
\begin{equation}
  \begin{split}
  \label{eq:SBE}
  \|f\|_{\dot{F}^{t}_{\vec{r},q}}\le c\|f\|_{\dot{F}^s_{\vec{p},\infty}} 
  &\quad\text{for}\quad t<s,\quad r_1\ge p_1,\dots, r_n\ge p_n,
\\
  &\quad\text{and}\quad 
  t-\frac1{r_1}-\dots-\frac1{r_n}=s-\frac1{p_1}-\dots-\frac1{p_n}.
\end{split}
\end{equation}
\end{proposition}

This inequality may be obtained from the arguments given for the inhomogeneous spaces in
\cite{MR2319603}. Indeed, this proof requires no essential changes if one only observes
the following interpolation inequality, valid for $-\infty<s_1<s_2<\infty$, $0<\theta<1$, $0<q<\infty$,
\begin{equation}
  \|2^{(s_1\theta+s_2(1-\theta))j}x_j\|_{\ell_q}\le c\|2^{s_1j}x_j\|_{\ell_\infty}^{\theta} \|2^{s_2j}x_j\|_{\ell_\infty}^{1-\theta}.  
\end{equation}
This was proved for sequences $\{x_j\}$ with $j\in\bN$ by Brezis and Mironescu~\cite{MR1877265}, but their
argument based on monotonicity extends verbatim to sequences having $j\in\bZ$. 

\begin{remark} \label{emb-rem}
  Embeddings of the homogeneous spaces $\dot F^s_{\vec p,q}$ constitute a main area where their
  theory differs from that of the corresponding inhomogeneous spaces $F^s_{\vec p,q}$. For example
  \begin{equation}
    \dot F^s_{\vec p,q}\not\subset \dot F^{s-\varepsilon}_{\vec p,u} 
  \qquad\text{for $\varepsilon>0$ and $0< u\le \infty$}.
  \end{equation}
This addendum to \eqref{qu-emb} comes from the fact that $2^{-\nu\varepsilon}$ is unbounded for
$\nu\to-\infty$. To give a proof we assume 
$\dot F^s_{\vec p,q}\subset \dot F^{s-\varepsilon}_{\vec p,u}$. Then the inclusion operator is bounded, i.e.\
$\|f\|_{\dot{F}^{s-\varepsilon}_{\vec{p},u}}\le c\|f\|_{\dot{F}^s_{\vec{p},q}}$ for some $c>0$;
this follows from the closed graph theorem using \eqref{eq:ZZ'} and the completeness of the
$\dot F^s_{\vec p,q}$ shown in Corollary~\ref{Fspq-cor} below. 
If $q=\infty$ this means that for $t=s-\varepsilon$ and $\vec r=\vec p$ the conditions in \eqref{eq:SBE} are violated,
so a contradiction is obtained at once from the next remark.
Cases with $0<q<\infty$ can be reduced to this using \eqref{eq:SBE} on the right-hand side of the inequality.
\end{remark}

\begin{remark} \label{SBE-rem}
That the Sobolev inequality in \eqref{eq:SBE} is stated under sharp conditions on the
parameters can be deduced from the auxilliary Schwartz functions $\theta_k$ and $\rho_k$ in Lemma~4.1 of
\cite{MR1355014}: First we arrange that $\varphi=1$ where $K_1\le|\xi|\le K^1$; then we take $\hat\theta\in
C^\infty_0$ to have a small support in this annulus and translate it by $\xi=(2^k,0,\dots,0)$ to get $\theta_k$
for $k\ge1$;
hence $\|\theta_k\|=2^{sk}\|\theta\|_{\vec p}$ in every generic $\dot F^s_{\vec p,q}$, so $t\le s$ follows from the
Sobolev inequality, and if $t=s$ the embedding is a tautology for $u=\infty$ or else impossible since
$\ell_u\hookrightarrow \ell_\infty$. 
The functions $\rho_k(x)=\rho(2^k x)$ are made by dilation of a function with spectrum in the
annulus, hence have the generic $\dot F$-norms 
$\|\rho_k\|= \|\rho\|_{\vec p}2^{k(s-\frac1{p_1}-\dots-\frac1{p_n}})$, so for
  $k\to\pm \infty$ one finds $s-\frac1{p_1}-\dots-\frac1{p_n}\gtrless
  t-\frac1{r_1}-\dots-\frac1{r_n}$; whence the equality stated in \eqref{eq:SBE}.
Finally $r_j\ge p_j$ must hold, for one can localise the support of $\hat\rho$ 
to a small neighbourhood of $\xi_j=0$ inside the annulus and take $\tilde\rho_k$ to be like $\rho_k$ except that the
dilation is omitted in the variable $x_j$ so that 
$\|\tilde\rho_k\|= \|\rho_k\|2^{k\frac1{p_j}}$ for all $k\ge1$; then
the Sobolev inequality and the equality obtained above entail that
$1\le c2^{k(\frac1{p_j}-\frac1{r_j})}$ for all $k\ge1$, whence $r_j\ge p_j$. 
\end{remark}

Following Franke~\cite{MR847350} we establish the Fatou property, namely
that the centered balls of $\dot F^s_{\vec p,q}$ are stable under sequential convergence in $\cZ'$:
\begin{lemma}  \label{Fatou-lem}
  If $f^{(m)}$, $m\in\bN$, satisfy that $f^{(m)}\to f$ in the $\operatorname{w}^*$-sense in
  $\cZ'$, then
  \begin{equation}
    \label{eq:Fatou}
      \|f\|_{\dot F^s_{\vec p,q}} \le \liminf_m\|f^{(m)}\|_{\dot F^s_{\vec p,q}}.
  \end{equation}
\end{lemma}
\begin{proof}
  Using $\varphi_\nu$ from the $\dot F^s_{\vec p,q}$-norm 
  we set $f^{(m)}_\nu(x)=\langle f^{(m)},\overline{\varphi}_\nu(x-\cdot)\rangle$ and
  define $f_\nu=\varphi_\nu*f(x)$ analogously. Since $\varphi$ is in $\cZ\subset\cS$, this can be
  read as a scalar product  on $\cS'\times\cS$, so $f^{(m)}_\nu$ and $f_\nu$ are
  $C^\infty$-functions by the Paley--Wiener--Schwartz Theorem. Clearly
  $f^{(m)}_\nu(x)\to f_\nu(x)$ pointwise for $m\to\infty$, so we obtain
\begin{equation}
  \begin{split}
     \|f\|_{\dot F^s_{\vec p,q}}= \|\liminf_m f^{(m)}\|_{\dot F^s_{\vec p,q}}
          &\le \| \liminf_m (\sum_\nu 2^{s\nu q}|f^{(m)}_\nu(x)|^q)^{\frac1q} \|_{\vec p} 
\\
          &\le \liminf_m \| (\sum_\nu 2^{s\nu q}|f^{(m)}_\nu(x)|^q)^{\frac1q} \|_{\vec p}
  \end{split}
  \end{equation}
by using Fatous's lemma for the counting measure on $\bZ$ and $n$ times for the Lebesgue measure
(and that $(\liminf_m|x_m|)^t=\liminf_m|x_m|^t$ for $t>0$). 
\end{proof}

The discrete analogue of Triebel--Lizorkin spaces is the space of sequences that we introduce here:

\begin{definition}\label{def:dT-L}
For $s\in\mathbb{R},\;\vec{p}=(p_1,\dots,p_n),$ with $0<p_1,\dots,p_n<\infty,\;0<q\leq \infty$,
the sequence space $\dot{f}^s_{\vec{p},q}$ consists of all complex-valued $a=\{a_Q\}_{Q\in\mathcal{Q}}$ such that 
\begin{equation}\label{dTLnorm}
\|a\|_{\dot{f}^s_{\vec{p},q}}:=\Big \| \Big(\sum_{Q\in\mathcal{Q}} \big
(|Q|^{-s/n}|a_Q|\widetilde{\mathds{1}}_Q \big)^q \Big)^{1/q} \Big \|_{\vec{p}}<\infty, 
\end{equation}
where $\widetilde{\mathds{1}}_Q(x)=|Q|^{-1/2} \mathds{1}_Q(x)$, with $\mathds{1}_Q$ denoting the characteristic
function of the cube $Q$.
\end{definition}

The sum over $Q\in \cQ$ should be understood as the unambiguous expression
\begin{equation} \label{eq:dLWP}
  \Big(\sum_{\nu\in\bZ}\sum_{\ell(Q)=2^{-\nu}} \big (\frac{|a_Q|}{|Q|^{s/n}}\widetilde{\mathds{1}}_Q \big)^q \Big)^{1/q}
   = \Big(\sum_{\nu\in\bZ}\big (\sum_{\ell(Q)=2^{-\nu}} \frac{|a_Q|}{|Q|^{s/n}}\widetilde{\mathds{1}}_Q \big)^q\Big)^{1/q}.
\end{equation}
Indeed, the inner sum is just a convenient notation for a piecewise constant function on $\bR^n$,
equal to $|a_Q|2^{\nu(s+n/2)}$ in each $Q$; the identity is due to the disjoint cubes.
Accordingly, for $q=\infty$ the $\ell_q$-norm is replaced by the supremum over $\nu$ only. 
For $q=2$ the quantity \eqref{eq:dLWP} is known as the discrete Littlewood-Paley expression.

At the level of the sequence space, completeness is rather easily obtained:

\begin{lemma} \label{lem:f_complete}
  The sequence space $\dot{f}^{s}_{\vec{p},q}$ is a quasi-Banach space, and for $q<\infty$ the
  sequences of finite support form a dense subspace.
\end{lemma}
\begin{proof}
  Given a Cauchy sequence $a^{(k)}$ of elements in $\dot{f}^{s}_{\vec{p},q}$ there is to
  $\varepsilon>0$ some $K$ such that for $k,m\ge K$,
  \begin{equation}
    \Big \| \Big(\sum_{Q\in\mathcal{Q}} \big
    (|Q|^{-s/n}|a^{(k)}_Q-a^{(m)}_Q|\widetilde{\mathds{1}}_Q  \big)^q
    \Big)^{1/q} \Big \|_{\vec{p}} <\varepsilon.
  \end{equation}
Keeping a single summand indexed by $Q_0$ yields
$|a^{(k)}_{Q_0}-a^{(m)}_{Q_0}|\|\widetilde{\mathds{1}}_{Q_0}\|_{\vec{p}}<\varepsilon$, so 
 it is seen that $a^{(k)}_{Q_0}$ converges to some $a_{Q_0}\in\bC$ for $k\to\infty$. By 
 $(n+1)$-fold application of Fatou's lemma to the limit 
$m\to\infty$ in the above, one finds that $\|a^{(k)}-a\|_{\dot{f}^s_{\vec{p},q}}\le\varepsilon$ for
$k\ge K$; whence $a\in\dot{f}^s_{\vec{p},q}$ and completeness follows. 

When $q<\infty$ and  $a\in\dot{f}^s_{\vec{p},q}$, then any chain of finite subsets
$\mathcal{Q}_1\subset\mathcal{Q}_2\subset\dots$ such that $\mathcal{Q}=\bigcup_N \mathcal{Q}_N$ gives
sequences $a\mathds{1}_{\mathcal{Q_N}}$ of finite support. That
$a\mathds{1}_{\mathcal{Q_N}}-a\to0$ in $\dot{f}^s_{\vec{p},q}$ for $N\to\infty$ follows by repeated
dominated convergence.
\end{proof}

For completeness of the function space $\dot F^s_{\vec p,q}$ the reader is referred to
Corollary~\ref{Fspq-cor}, where this property is carried over from Lemma~\ref{lem:f_complete} as an
easy consequence of the main theorem. 

As a comment on the sequences in the space $\dot f^s_{\vec p,q}$, note that by dropping all terms
but one in the norm we get $|a_Q||Q|^{-s/n-1/2}\| \mathds{1}_Q\|_{\vec p}\le \| a\|_{\dot f^s_{\vec p,q}}=:C_a$. 
Thus we obtain the crude estimate
\begin{equation} \label{aQ-est}
  |a_Q| \le C_a |Q|^{s/n+1/2} \ell(Q)^{-\frac1{p_1}-\dots-\frac1{p_n}} 
        = C_a 2^{-\nu(s+n/2-(\frac1{p_1}+\dots+\frac1{p_n}))}. 
\end{equation} 
For $\nu\to\infty$ this is only a decay estimate in case
$s+n/2>\frac1{p_1}+\dots+\frac1{p_n}$; whereas for $\nu\to-\infty$ it is only for 
$s+n/2<\frac1{p_1}+\dots+\frac1{p_n}$ that the above gives decay.

\section{Proof of the Main Result} \label{TSI-sect}
In this section we derive the main theorem on the boundedness of $S_\varphi$ and $T_\psi$
together with the identity $T_\psi\circ S_\varphi = I$. Namely, we prove the following:

\begin{theorem} \label{th:main}
Let $s\in\mathbb{R}$ with $0<p_1,\dots,p_n<\infty$ and $0<q\leq \infty$.
For any admissible functions $\varphi, \psi$ the $\varphi$-transform $S_\varphi\colon
\dot{F}^{s}_{\vec{p},q}\rightarrow \dot{f}^{s}_{\vec{p},q}$ and the inverse $\varphi$-transform 
$T_\psi\colon \dot{f}^{s}_{\vec{p},q}\rightarrow\dot{F}^{s}_{\vec{p},q}$ are bounded operators. 

Furthermore, when the reconstruction identity \eqref{phi4} is satisfied by $\varphi$ and $\psi$, then
\begin{equation}
  T_\psi\circ S_\varphi f= f \qquad\text{for every $f\in\dot{F}^{s}_{\vec{p},q}$}.
\end{equation}
In particular $\|\cdot\|_{\dot{F}^{s}_{\vec{p},q}}\approx \|S_\varphi(\cdot)\|_{\dot{f}^{s}_{\vec{p},q}}$ 
and $S_\varphi(\dot F^s_{\vec p,q})$ is complemented; cf.\ \eqref{directsum-id}.
\end{theorem} 

To avoid an excess of concurrent estimates, and to crystallise results of independent interest, we will split the proof of Theorem \ref{th:main} into a number of steps.

\subsection{The synthesis operator $T_\psi$}
First we show that the crude estimates \eqref{aQ-est} suffice for the following basic result 
which extends Proposition~\ref{TS-prop} on $D(T_\psi)$ to the full sequence space:

\begin{proposition} \label{fspq-prop}
  For any admissible $\psi\in\cS$ the synthesis operator $T_\psi$ from
Theorem~\ref{Tpsi-thm} is defined on $\dot f^s_{\vec p,q}$ for any
$s\in\bR$, $p_1,\dots,p_n\in\,]0,\infty[\,$ and $0<q\le\infty$.
\end{proposition}
\begin{proof}
  To verify the integrability condition in Theorem~\ref{Tpsi-thm} for a given  $a\in \dot f^s_{\vec p,q}$, we shall show that the 
  series $\sum_Q|a_Q\langle \psi_Q,\phi\rangle|$ is finite for arbitrary $\phi\in\cZ$. 
   We invoke \eqref{aQ-est} to make a comparison with $S_++S_-$, whereby 
  \begin{equation}
    S_\pm = \sum_{\mu\gtrless 0}\sum_{k\in\bZ^n} |\langle \psi_Q,\phi\rangle| C_a 2^{-\mu(s+n/2-(\frac1{p_1}+\dots+\frac1{p_n}))}.
  \end{equation}
Here we use that $\langle \psi_Q,\phi\rangle=\psi_Q*\tilde\phi(0)$, where the normalisation by
$|Q|^{-1/2}$ in $\psi_Q$ is cancelled by the factor $2^{-n\mu/2}=|Q|^{1/2}$ above. 
Now Lemma~\ref{lem:p*p} applies for $\nu=0$, and for $\mu\ge0$ the first estimate \eqref{eq:p*p} gives 
the comparison 
  \begin{equation} \label{S+est}
   S_+\le \sum_{\mu\ge 0}\sum_{k\in\bZ^n} C_a 2^{-\mu(s-(\frac1{p_1}+\dots+\frac1{p_n}))}C_N2^{(N-n)\mu}(1+2^\mu|x_Q|)^{-N}.
  \end{equation}
Since $2^\mu x_Q=k$ in this sum, we use \eqref{kN-est} for $N>n$.
So in case $s>\frac1{p_1}+\dots+\frac1{p_n}$ we obtain for sufficiently small
$N>n$ that $S_+<\infty$.

For $s\le \frac1{p_1}+\dots+\frac1{p_n}$ there is a reinforcement in terms of the
estimate \eqref{eq:p*pM}, which we may apply as $\psi$ is admissible, hence fulfils the moment
condition of any order $M\in\bN_0$. Therefore one can replace $C_N2^{(N-n)\mu}$ in \eqref{S+est} by
$C'_N2^{(N-n-M)\mu}$ for $M$ so large that $s-\frac1{p_1}-\dots-\frac1{p_n}-(N-n)+M>0$, 
implying that $S_+$ is finite.

For $S_-$ the argument leading to inequality \eqref{S+est} gives a slightly simpler estimate, since
\eqref{eq:p*p} now applies for $\mu<0=\nu$; namely
  \begin{equation} \label{S-est}
   S_-\le \sum_{\mu=-\infty}^{-1} C_a 2^{-\mu(s-(\frac1{p_1}+\dots+\frac1{p_n}))} C_N\sum_{k\in\bZ^n}(1+|k|)^{-N}.
  \end{equation}
Clearly the right-hand side is finite for $s-\frac1{p_1}-\dots-\frac1{p_n}<0$ and $N>n$; cf.~\eqref{kN-est}. 

In the complementing region with $s\ge \frac1{p_1}+\dots+\frac1{p_n}$ we note that $\phi$ as a member
of $\cZ$ also fulfils the moment condition of any order $M\in\bN_0$, so \eqref{eq:p*pM'}
allows us to replace $C_N$
in \eqref{S-est} by $C''_N2^{M\mu}$ for some $N>n$, $M\in\bN$ so large that $s-\frac1{p_1}-\dots-\frac1{p_n}-M<0$.
Whence $S_-<\infty$ also in this case.

Altogether this shows that every $a\in\dot f^s_{\vec p,q}$ belongs to 
$\bigcap_{\phi\in\cZ}L_1(\cQ,\langle\psi_Q,\phi\rangle d\tau_{1+n})$, which by \eqref{eq:1} is the domain of $T_\psi$.
\end{proof}

To obtain a more refined estimate, we will need the following estimate, closely related to the 
Peetre-type maximal inequality \eqref{M3}; cf.\ Appendix~\ref{maximal_lemma} for a proof.

\begin{lemma} \label{l:star} Let $0<t\le1$, $\tau>n/t$ and $\mu\in\mathbb{Z}$. For any sequence $(a_P)_{P\in\cQ}$ we have 
\begin{equation}
  \sum_{\ell(P)=2^{-\mu}}\frac{|a_P|}{\big(1+2^{\min(\mu,\nu)}|x_P-x_Q|\big)^{\tau}}
 \le c 2^{\frac{n}{t}(\mu-\nu)_+}\big(M_n\cdots M_1(\sum_{\ell(P)=2^{-\mu}}|a_P|\mathds{1}_P)^t\cdots\big)^{\frac1t}(x)
\end{equation}
whenever $x\in Q $ with $\ell(Q)=2^{-\nu}$ for some $\nu\in\mathbb{Z}$.
\end{lemma}

We note a few consequences of Lemma~\ref{l:star} for later use.
They concern the transformation of a given sequence $a$ into $a^*=\{a^*_Q\}$ given by
\begin{equation}\label{s1}
  a^*_Q=\sum_{\ell(P)=2^{-\nu}}|a_P|(1+2^{\nu}|x_P-x_Q|)^{-\tau},
\end{equation}
for every $Q\in\mathcal{Q}$ with $\ell(Q)=2^{-\nu}$, $\nu\in\mathbb{Z}$, and some $\tau>n/t$, for
$0<t\le1$. 
Here Lemma~\ref{l:star} applied for $\mu=\nu$ gives, for every $x\in Q$,
\begin{equation}\label{s2}
  a^*_Q\le c\Big(M_n\cdots M_1\big(\sum_{\ell(P)=\ell(Q)}|a_P|\mathds{1}_P\big)^{t}\cdots\Big)^{1/t}(x).
\end{equation}
Since the set of all dyadic cubes with the same side-length is a disjoint partition of
$\mathbb{R}^n,$ it is clear from (\ref{s2}) that on $\bR^n$,
\begin{equation}\label{s3}
  \sum_{\ell(Q)=2^{-\nu}}a^{*}_{Q}\mathds{1}_Q(x)
  \le c\Big(M_n\cdots M_1\big(\sum_{\ell(P)=2^{-\nu}}|a_P|\mathds{1}_P\big)^{t}\cdots\Big)^{1/t}(x).
\end{equation}

We can now prove the following refinement  of Proposition~\ref{fspq-prop}: 

\begin{proposition} \label{Tpsi-prop}
  When  $\psi$ is admissible, then the synthesis operator $T_\psi$ from Theorem~\ref{Tpsi-thm} is a
  bounded linear map $\dot f^s_{\vec p,q}\to \dot F^s_{\vec p,q}$.
\end{proposition}

\begin{proof} 
Let $a=\{a_Q\}$ be a sequence in $\dot{f}^{s}_{\vec{p},q}$ of finite support and
$f:=T_{\psi}a=\sum_{Q}a_Q\psi_Q$. 
By the finiteness and \eqref{phi5}, we obtain with $h=\log_2(K^0/K_0)$ 
\begin{equation}\label{p1}
  \varphi_{\nu}*f(x)=\sum_{J}a_J\varphi_{\nu}*\psi_J=\sum_{\nu-h\le\mu\le\nu+h}\sum_{\ell(J)=2^{-\mu}}a_{J}\varphi_{\nu}*\psi_{J}(x)
\end{equation}
for any $x\in\bR^n$, $\nu\in\mathbb{Z}$.
Using the basic estimate \eqref{eq:p*p} from Lemma~\ref{lem:p*p}, and the support relation
\eqref{phi6}, we see that for $\tau>0$
\begin{equation}
\label{p2}
  |\varphi_{\nu}*\psi_{J}(x)|\le c|J|^{-1/2}(1+2^{\mu}|x-x_{J}|)^{-\tau}.
\end{equation}
Now we fix $\tau$ so large that $\tau>n/t$ and $0<t<\min(1,p_1,\dots,p_n,q)$. Hence
\begin{equation}\label{p3}
  |Q|^{-s/n}|\varphi_{\nu}*f(x)|\le c\sum_{|\mu-\nu|\le h}\sum_{\ell(J)=2^{-\mu}}|a_{J}||J|^{-s/n-1/2}(1+2^{\mu}|x-x_{J}|)^{-\tau}.
\end{equation}

For every  $x\in\mathbb{R}^n$ and $\mu\in\mathbb{Z},$ there exists a unique $P_0\in\mathcal{Q}$
such that $x\in P_0$ and $\ell(P_0)=2^{-\mu}.$ Then it holds that 
$1+2^{\mu}|x_{P_0}-x_J|\le1+2^{\mu}(|x-x_{P_0}|+|x-x_J|)
\le1+\sqrt{n}+2^{\mu}|x-x_J|\le c_n(1+2^{\mu}|x-x_J|)$ and thus from \eqref{p3}
\begin{equation}  \label{p4}
  |Q|^{-s/n}|\varphi_{\nu}*f(x)|\le c\sum_{|\mu-\nu|\le h}\sum_{\ell(J)=2^{-\mu}}|a_{J}||J|^{-s/n-1/2}(1+2^{\mu}|x_{P_0}-x_{J}|)^{-\tau}.
\end{equation}
For simplicity we introduce $b=\{b_J\}$ with $b_J=a_J|J|^{-s/n-1/2}$ for every
$J\in\mathcal{Q}$,  and then using \eqref{p4}, \eqref{s1} and \eqref{s3} we derive
\begin{equation}
  \begin{split}
  2^{\nu s}|\varphi_{\nu}*f(x)|&\le c\sum_{|\mu-\nu|\le h}b^*_{P_0}\mathds{1}_{P_0}(x)
  \le c\sum_{\mu=\nu-1}^{\nu+1}\sum_{\ell(P)=2^{-\mu}}b^*_P\mathds{1}_P(x)
\\
  &\le c\sum_{|\mu-\nu|\le  h}\Big(M_n\cdots M_1\big(\sum_{\ell(P)=2^{-\mu}}|b_P|\mathds{1}_P\big)^{t}\cdots\Big)^{1/t}(x).
  \end{split}
\label{p5}    
\end{equation}
Applying \eqref{p5} in the $\dot F^s_{\vec p,q}$-norm of $f$, the quasi-triangle inequality gives
\begin{equation}
  \begin{split}
  \|f\|_{F_{\vec{p},q}^s}&=\Big\|\Big(\sum_{\nu\in\mathbb{Z}} \Big(2^{\nu s}|\varphi_{\nu}* f(\cdot)|\Big)^q\Big)^{\frac1q}\Big\|_{\vec{p}}
\\
  &\le c\Big\|\Big(\sum_{\nu\in\mathbb{Z}} \Big(\sum_{|\mu-\nu|\le h}\Big(M_n\cdots
  M_1\big(\sum_{\ell(P)=2^{-\mu}}|b_P|\mathds{1}_P\big)^{t}\cdots\Big)^{\frac1t}\Big)^q\Big)^{\frac1q}\Big\|_{\vec{p}} 
\\
  &\le c(1+2h)\Big\|\Big(\sum_{\mu\in\mathbb{Z}}\Big(\Big(M_n\cdots
   M_1\big(\sum_{\ell(P)=2^{-\mu}}|b_P|\mathds{1}_P\big)^{t}\cdots\Big)^{\frac1t}\Big)^q\Big)^{\frac1q}\Big\|_{\vec{p}}.
  \end{split}
\end{equation}
So by invoking the maximal inequality \eqref{M2}, the definition of $b_P$ above, and that the sum
over $P$ contains a single term at each fixed $x$, respectively, we obtain 
\begin{equation}
  \begin{split}
  \|f\|_{F_{\vec{p},q}^s} &\le c\Big\|\Big(\sum_{\mu\in\mathbb{Z}}\Big(\sum_{\ell(P)=2^{-\mu}}|b_P|\mathds{1}_P\Big)^q\Big)^{1/q}\Big\|_{\vec{p}}
\\
  &=c\Big\|\Big(\sum_{\mu\in\mathbb{Z}}\Big(\sum_{\ell(P)=2^{-\mu}}|P|^{-s/n}|a_P|\widetilde{\mathds{1}}_P\Big)^q\Big)^{1/q}\Big\|_{\vec{p}}
    \le c\|a\|_{\dot{f}^{s}_{\vec{p},q}}.
  \end{split}
\end{equation}

Thus $T_\psi\colon \dot f^s_{\vec p,q}\to \dot F^s_{\vec p,q}$ has been shown to be bounded on the
subspace of sequences of finite support. 
For any given sequence $a$ in $\dot f^s_{\vec p,q}$ we recall from Proposition~\ref{fspq-prop} that
$T_\psi a$ is defined. This is exploited by choosing approximations $a^{(m)}$ having finite, increasing and exhausting
supports, so that by the last part of Theorem~\ref{Tpsi-thm} we have in the $\operatorname{w}^*$-topology of $\cZ'$ that
\begin{equation}
  T_\psi a=\lim_{m\to\infty} T_\psi a^{(m)}.
\end{equation}
Moreover, the boundedness and truncation give the uniform bound
\begin{equation}
  \| T_\psi a^{(m)}\|_{\dot F^s_{\vec p,q}} \le c\|a^{(m)}\|_{\dot f^s_{\vec p,q}}\le c\|a\|_{\dot f^s_{\vec p,q}}.
\end{equation}
In view of these facts, the Fatou property
yields that also $\| T_\psi a\|_{\dot F^s_{\vec p,q}} \le c\|a\|_{\dot f^s_{\vec p,q}}$; cf.\  Lemma~\ref{Fatou-lem}.
Hence $T_\psi$ is bounded on the full sequence space $\dot f^s_{\vec p,q}$. 
\end{proof}

\subsection{The analysis operator $S_\varphi$}
For the analysis operator $S_\varphi$ we adopt a 2-step procedure. For
clarity we write $\dot F^s_{\vec p,q}(\varphi)$ to emphasize that the Triebel--Lizorkin space, or
its norm, has been defined by means of the admissible function $\varphi$.

\begin{proposition}
  \label{Sphi-prop}
  If  $\varphi$ is admissible,  then $S_\varphi$  has the property of boundedness
 \begin{equation}
   \|S_\varphi f\|_{\dot f^s_{\vec p,q}} \le c\|f\|_{\dot F^s_{\vec p,q}(\tilde\varphi)}
  \qquad\text{for every $f\in \dot F^s_{\vec p,q}(\tilde\varphi)$}.
 \end{equation}
\end{proposition}

\begin{proof}
Let $f$ be arbitrary in $\dot{F}^{s}_{\vec{p},q}(\tilde\varphi)$ for the given admissible function
$\tilde\varphi$. 
For $Q\in\mathcal{Q}$ with $\ell(Q)=2^{-\nu}$, $\nu\in\mathbb{Z}$,
we obtain as in \eqref{eq:phiQ} that 
\begin{equation*}
  (S_{\varphi}f)_Q=\langle f,\varphi_Q\rangle=|Q|^{1/2}\tilde{\varphi}_{\nu}*f(x_Q).
\end{equation*}
Therefore we crudely get for any $t>0$, since $1+2^{\nu}|x_Q-x|\le 1+\sqrt{n}$ for $x\in Q$,  
\begin{equation}
  \begin{split}
    \sum_{\ell(Q)=2^{-\nu}} \frac{|(S_{\varphi}f)_Q|^q}{|Q|^{sq/n}}\widetilde{\mathds{1}}_Q(x)
  &\le c\sum_{\ell(Q)=2^{-\nu}}\frac{2^{\nu qs}|\tilde{\varphi}_{\nu}*f(x_Q)|^q}{(1+2^{\nu}|x_Q-x|)^{nq/t}}\mathds{1}_Q(x)
\\
  &\le c\sum_{\ell(Q)=2^{-\nu}}\left(\sup_{y\in \bR^n}
         \frac{2^{\nu s}|\tilde{\varphi}_{\nu}*f(y)|}{(1+2^{\nu}|y-x|)^{n/t}}\right)^{q}\mathds{1}_Q(x).
  \end{split}
\label{p6}
\end{equation}
Fixing $t<\min(p_1,\dots,p_n,q)$ 
we may apply the maximal inequality \eqref{M3}, as \eqref{phi5} entails 
$\text{supp} (\widehat{\tilde\varphi_{\nu}*f})=\text{supp}(\widehat{\tilde\varphi_{\nu}}\hat{f})
\subset [-K^02^{\nu},K^02^{\nu}]^n$, which is contained in $[-2^{v'},2^{v'}]^n$ for some $\nu'>\nu$.
Thus we get, for $x\in\mathbb{R}^n$,
\begin{equation}  \label{p7}
  \sum_{\ell(Q)=2^{-\nu}} \frac{|(S_{\varphi}f)_Q|^q}{|Q|^{sq/n}} \widetilde{\mathds{1}}_Q(x)
  \le c\left(\Big(M_n\cdots M_1\big(2^{\nu
      s}|\tilde{\varphi}_{\nu}*f|\big)^{t}\cdots\Big)^{1/t}(x)\right)^q
  \mathds{1}_{\bR^n}(x).
\end{equation}
We pass to the discrete Triebel--Lizorkin norm of $S_{\varphi}f$ by calculating the norm of 
$\ell_q$ with respect to $\nu\in\bZ$ and that of
$L_{\vec p}(\bR^n)$ on both sides of \eqref{p7}. So by using the maximal inequality \eqref{M2} we obtain
\begin{equation}
    \|S_{\varphi}f\|_{f_{\vec{p},q}^s} 
   \le c\Big\|\Big(\sum_{\nu\in\mathbb{Z}}\Big(M_n\cdots M_1\big(2^{\nu s}|
    \tilde{\varphi}_{\nu}*f|\big)^t(\cdot)\Big)^{q/t}\Big)^{1/q}\Big\|_{\vec{p}} 
   \le c\|f\|_{\dot{F}^{s}_{\vec{p},q}(\tilde\varphi)}.
\end{equation}
This proves the stated inequality for $S_\varphi$.
\end{proof}

The above result could be improved, since it would be natural to replace
$\tilde\varphi$ by $\varphi$ in the inequality---or indeed to replace it by an arbitrary admissible
function $\Phi$, so that boundedness of $S_\varphi$ would be decoupled from the choice of norm on
$\dot F^s_{\vec p,q}$. 

The remedy lies in a classical argument from the $\varphi$-transform theory.
But it is a main point that heuristic use of $T_\psi a$ as a ``sum'' should be replaced by rigorous
reference to the definition by the Pettis integral, so we proceed with diligence:

Let $\varphi$, $\Phi$ be two arbitrary admissible functions. Then there are admissible functions $\psi$,
$\Psi$ such that each couple $(\varphi,\psi)$, $(\Phi,\Psi)$ satisfies \eqref{phi4}. This implies that
\begin{equation}
  \|f\|_{\dot{F}^s_{\vec{p},q}(\varphi)} = \|\sum_Q (S_{{\Phi}}f)_Q{\Psi}_Q\|_{\dot{F}^s_{\vec{p},q}(\varphi)}
  \le c\|S_{{\Phi}}f\|_{\dot{f}^s_{\vec{p},q}}\le c\|f\|_{\dot{F}^s_{\vec{p},q}(\tilde{\Phi})}.
\end{equation}
Indeed, we may substitute $f= \sum_Q (S_{{\Phi}}f)_Q{\Psi}_Q$ in the
first norm, since we proved in all details that $T_\Psi$ is a left-inverse of $S_\Phi$ on
$\cS'/\cP$; cf.\ Proposition~\ref{l:identity}. 
And the first inequality above holds, since $T_\Psi$ is everywhere defined and bounded according to
Proposition~\ref{Tpsi-prop} (no connection between the two admissible functions $\varphi$, $\Psi$ is required). 
Finally Proposition~\ref{Sphi-prop} suffices for the last inequality.

Substituting by the admissible functions $\tilde\Phi$ and $\tilde\varphi$, it is seen at once that also 
$\|f\|_{\dot{F}^s_{\vec{p},q}(\tilde{\Phi})} \le c\|f\|_{\dot{F}^s_{\vec{p},q}(\varphi)}$ holds. 
Consequently either both or none of the Triebel--Lizorkin norms are finite on any given 
$f\in\cS'/\cP$. Therefore $\dot F^s_{\vec p,q}(\varphi)$ equals the space 
$\dot F^s_{\vec   p,q}(\tilde\Phi)$; and their norms are equivalent in view of the just shown
inequalities. 

Hence Proposition~\ref{Sphi-prop} can be sharpended to boundedness of $S_\varphi$ with respect to
any norm on $\dot F^s_{\vec p,q}$. Thus we have completed the 2-step procedure; the outcome may be stated as follows:

\begin{proposition}
  \label{Fnorm-prop}
When $\varphi$, $\Phi$ are admissible for the same set of constants in \eqref{phi2}--\eqref{phi01},
then the induced Triebel--Lizorkin spaces coincide and the corresponding norms are equivalent, i.e.\
\begin{equation}
  \|\cdot\|_{\dot F^s_{\vec p,q}(\varphi)}\approx \|\cdot\|_{\dot F^s_{\vec p,q}(\Phi)}.  
\end{equation}
Moreover, $S_\varphi\colon \dot F^s_{\vec p,q}\to \dot f^s_{\vec p,q}$ is a bounded operator.
\end{proposition}

\subsection{Proof of Theorem \ref{th:main}}
The boundedness of $S_\varphi$ and $T_\psi$ has been obtained in Proposition~\ref{Fnorm-prop} and
\ref{Tpsi-prop}, respectively.

Combining the boundedness and the identity $T_\psi\circ S_\varphi=I$, cf.\
Proposition~\ref{l:identity}, one gets at once that  
$\| S_\varphi f\|_{\dot f^s_{\vec p,q}}$ is equivalent to the norm on $\dot F^s_{\vec p,q}$: 
that is, for certain constants $B\ge 1\ge A>0$ we have the classical inequalities
\begin{equation}
  A\|f\|_{\dot F^s_{\vec p,q}}\le \| S_\varphi f\|_{\dot f^s_{\vec p,q}}\le B \|f\|_{\dot F^s_{\vec p,q}}.  
\end{equation}
Secondly, $P:=S_\varphi\circ T_\psi$ is a continuous idempotent, and as such projects onto
its range $R(S_\varphi)$ along the nullspace of $T_\psi$; and the range is closed (cf.\ the proof
of Corollary~\ref{Fspq-cor} below). More precisely, $S_\varphi(\dot F^s_{\vec p,q})$ is a
complemented subspace, i.e.\ with a direct sum of (quasi-)Banach spaces,
\begin{equation}  \label{directsum-id}
   \dot f^s_{\vec p,q}=S_\varphi(\dot F^s_{\vec p,q})\oplus \{\, a\mid T_\psi a=0\,\}.  
\end{equation}
This concludes the proof of  Theorem \ref{th:main}.

Now it is straightforward to derive additional properties, as expected.

\begin{corollary} \label{Fspq-cor}
  $\dot{F}^s_{\vec{p},q}(\bR^n)$ is complete and the range
  $S_{\varphi}(\dot{F}^s_{\vec{p},q})$ is closed in $\dot{f}^s_{\vec{p},q}$.
\end{corollary}
\begin{proof}
  Every Cauchy sequence $f_k$ in $\dot{F}^s_{\vec{p},q}$ is sent into another Cauchy sequence
  $S_{\varphi}f_k$ by the bounded map $S_{\varphi}$; this has a limit $a=\{a_Q\}$ by the completeness
  of the sequence space shown in Lemma~\ref{lem:f_complete}. Using the boundedness of $T_\psi$ and
  that $S_{\varphi}$ is a right-inverse, cf.\ Theorem~\ref{th:main}, one obtains with limits in
  $\dot{F}^s_{\vec{p},q}$ that
  \begin{equation}
    T_\psi a=\lim_kT_\psi( S_{\varphi}f_k)=\lim_k f_k.
  \end{equation}
  This shows completeness. That $S_{\varphi}$ has closed range can be shown analogously, if one concludes by
  applying $S_{\varphi}$ to the above equation.
\end{proof}

\begin{corollary} \label{basis-cor}
  If $\varphi$, $\psi$ are admissible and fulfil the reconstruction identity \eqref{phi4} and the
  biorthogonality condition \eqref{biortho-id}, then the wavelets $\psi_Q$ give, through any numbering of
  the cubes $Q\in\cQ$, an unconditional basis for every $\dot F^s_{\vec{p},q}$ having $q<\infty$.
\end{corollary}
\begin{proof}
  According to Theorem~\ref{basis-thm} we have $f=\sum_{j=1}^\infty a_{Q_j}\psi_{Q_j}$ for
  $a=S_\varphi f$ in $\cZ'$, with unique scalars; by the continuous injection in 
  Proposition~\ref{ZZ'-prop}, the $f\mapsto a_{Q_j}$ are also continuous on  $\dot F^s_{\vec{p},q}$.
  To show convergence in the topology of $\dot F^s_{\vec{p},q}$, we
  introduce sequences of finite support $a^{(m)}=a\mathds{1}_{\{Q_1,\dots,Q_m\}}$, that 
  converge to $a$ in $\dot f^s_{\vec{p},q}$ for $q<\infty$; cf.\ the proof of
  Lemma~\ref{lem:f_complete}. Now, by the continuity of $T_\psi$ in Theorem~\ref{th:main}, 
  $\sum_{j=1}^m a_{Q_j}\psi_{Q_j}=T_\psi a^{(m)}$ converges for $m\to\infty$ 
  to $T_\psi a=T_\psi(S_\varphi f)=f$ in $\dot F^s_{\vec{p},q}$. Any other numbering gives the same
  result in view of Theorem~\ref{basis-thm}.
\end{proof}

\begin{remark} With Corollary~\ref{basis-cor} we just want to indicate how
  useful the rigorous definition of $T_\psi$ is in the discussion.
  Unconditional bases have also been emphasized by Triebel in a space-by-space approach in his
  works on wavelets, cf.\ \cite[3.1.3]{MR2250142} or \cite[Thm.~1.20]{MR2455724}, but without explicit proofs. 
\end{remark}

\appendix
\section{Convolution estimates} \label{p*p-app}
The proof of Lemma~\ref{lem:p*p} can be conducted as follows.
In the convolution integral one can exploit the classical estimate
  \begin{equation}
    (1+2^{\mu}|x-x_J|)^N\le (1+2^{\mu}|x-x_J-y|)^N(1+2^\mu|y|)^N,
  \end{equation}
which yields at once that the left-hand side of \eqref{eq:p*p} at most
equals $\sup (1+|\cdot|)^N|\psi|$ times the integral $\int
(1+2^\mu|y|)^N|2^{n\nu}\varphi(2^\nu y)|\,dy$. So for $\mu\le \nu$ one may set
$y=2^{-\nu}w$ to obtain \eqref{eq:p*p} with the constant 
$C_N=\sup(1+|\cdot|)^N|\psi| \cdot \int_{\bR^n} |\varphi(w)|(1+|w|)^N\,dw$.

For $\mu>\nu$ the convolution is written as $\int
\psi(z)\varphi(2^\nu(x-x_J-2^{-\mu}z))2^{n(\nu-\mu)}\,dz$.
So by taking $y=2^{-\mu}z$ in the above inequality, where $1\le 2^{\mu-\nu}$, one finds
\eqref{eq:p*p} with the extra factor $2^{(N-n)(\mu-\nu)}$, now for 
$C_N=\sup(1+|\cdot|)^N|\varphi| \cdot \int_{\bR^n} |\psi(w)|(1+|w|)^N\,dw$.

Elaborating on this, the moment condition yields that $\psi(z)$ vanishes by integration against the Taylor polynomial
of order $M$ of $\varphi$, whence
\begin{equation}
  \begin{split}
      \psi(2^\mu(\cdot-x_J))*\varphi_\nu(x)& = 
  \sum_{|\alpha|=M+1}
  \int 2^{n(\nu-\mu)}\psi(z)(-2^{\nu-\mu}z)^\alpha\varphi^{(\alpha)}(z)\,dz
\\
  \text{for}\quad \varphi^{(\alpha)}(z) &= \frac{M+1}{\alpha!}\!
  \int_0^1(1-\theta)^M\partial^\alpha\varphi(2^\nu(x-x_J)-\theta2^{\nu-\mu}z)\,d\theta.
  \end{split}
\end{equation}
For $y=\theta2^{-\mu}z$ the above inequality shows that one can take $C'_{N,M}=C'_{N,M}(\varphi,\psi)$  to be
$\sum_{|\alpha|=M+1}\sup(1+|\cdot|)^N|\partial^\alpha\varphi|/\alpha!$ times
$\int|\psi|(1+|\cdot|)^{N+M+1}\,dz$  for $\mu>\nu$. (Note that \eqref{eq:p*pM} is identical to \eqref{eq:p*p} for
$\mu\le \nu$.)

Finally, it is analogous to derive \eqref{eq:p*pM'} by letting $\psi$ and $\varphi$ change
roles, beginning with the convolution in the form 
$\int \psi(2^\mu(x-x_J-2^{-\nu}z))\varphi(z)\,dz$. This gives
$C''_{N,M}=C'_{N,M}(\psi,\varphi)$. The proof is complete.

\begin{remark}  \label{p*p-rem}
  The proof is valid \emph{verbatim} for $\mu,\nu\in\bR$, i.e.\ for dilation by $s=2^\nu>0$, $t=2^\mu>0$.
\end{remark}

\section{Proof of Lemma~\ref{l:star}}\label{maximal_lemma}
It suffices to prove the statement in Lemma~\ref{l:star} for any $\tau>n$ and $t=1$, for one can just replace $|a_P|$ by
$|a_P|^t$ and raise to the power $1/t$: on the left-hand side the fact that 
$\|\cdot\|_{\ell_1}\le\|\cdot\|_{\ell_t}$ for $0<t\le1$ gives the rest as $\tau/t>n/t$.

It also suffices to cover the case $\ell(Q)\ge \ell(P)$, i.e.\ $2^{-\nu}\ge2^{-\mu}$ or $\nu\le \mu$. In
fact, given $x\in Q$ for $\ell(Q)<\ell(P)$, then $x$ also belongs to a cube $J\in\cQ$ with $x_J=x_Q$ and
$\ell(J)=\ell(P)$, for which one then arrives at the inequality stated for $Q$.

We split the set of $P$'s as $\bigcup_{k\in\bN_0} \Omega_k$, whereby
\begin{align}
  \Omega_0&=\{\,P\in\cQ \mid \ell(P)=2^{-\mu} \text{ and } |x_P-x_Q|\le2^{-\nu}\,\},
\\
  \Omega_k&=\{\,P\in\cQ \mid \ell(P)=2^{-\mu} \text{ and }   2^{k-1-\nu}<|x_P-x_Q|\le2^{k-\nu}\,\},\quad k\ge1.
\end{align}
When $P\in\Omega_k$ we have $1+2^{\nu}|x_P-x_Q|>2^{k-1},$ so 
\begin{equation} \label{ls1}
  \sum_{\ell(P)=2^{-\mu}}|a_P|\Big(1+\frac{|x_P-x_Q|}{\ell(Q)}\Big)^{-\tau}
  \le 2^\tau\sum_{k=0}^{\infty}2^{-k\tau}\sum_{P\in\Omega_k}|a_P|.
\end{equation}
Because the $P$ in $\Omega_k$ are disjoint, and since $|P|=2^{-\mu n}$,
\begin{equation}  \label{ls2}
   \sum_{P\in\Omega_k}|a_P|=\int_{R}\big(\sum_{P\in\Omega_k}|a_P|2^{n\mu}\mathds{1}_P(y)\big)\,dy.
\end{equation}
Indeed, $\bigcup_{\Omega_k}P$ is contained in $R= x_Q+[-2^{k-\nu+1},2^{k-\nu+1}]^n$, for if $(y_1,\dots,y_n)\in P$ 
\begin{equation}
  |y_j-(x_Q)_j|\le |y_j-(x_P)_j|+|(x_P)_j-(x_Q)_j|\le 2^{-\nu} + 2^{k-\nu}\le2^{k-\nu+1}.  
\end{equation}
Since the side length of $R$ is $2^{k-\nu+2}$, every $x\in Q$ is in $R$, so by \eqref{M1}
\begin{equation}  \label{ls3}
  \sum_{P\in\Omega_k}|a_P| \le 4^n2^{(k-\nu+\mu)n} (M_n\cdots M_1\big(\sum_{\ell(P)=2^{-\mu}}|a_P|\mathds{1}_P\big)\dots)(x).
\end{equation}
Inserting \eqref{ls3} in \eqref{ls1}, it is straightforward to sum over $k$ and
obtain the inequality stated in Lemma~\ref{l:star}, since we assumed that $\tau>n$.

\begin{remark} \label{maximal-rem}
  As a corollary to the above proof, one finds the extension of Peetre's maximal inequality to
  the functions in Proposition~\ref{maximal-prop}
  (by the Paley--Wiener--Schwartz Theorem such functions are constant if they have
  compact spectrum as in \eqref{M3}). 
  Indeed, in \eqref{M3} each $y$ is in some $P$, and given $x\in Q$, the triangle inequality
  yields $1+2^\nu|x_P-x_Q|\le (1+2\sqrt n)(1+2^\nu|x-y|)$, hence an estimate from above by means of
  $\sum_{\ell(P)=2^{-\mu}}|a_P|\big(1+2^\nu|x_P-x_Q|\big)^{-\tau}$ for $\tau>n/t$.
  Taking $\mu=\nu$ in Lemma~\ref{l:star} there is a further estimate in terms of $c(M_n\dots M_1 |f|^t)^{1/t}(x)$, as
  claimed. 
\end{remark}

\section{Homogeneous Littlewood--Paley decompositions} \label{LP'-app}
  It is known that when $\hat\phi\in C_0^\infty(\bR^n)$ with $\supp\hat\phi\not\ni 0$, so that $\supp\hat\phi$
  is contained in an annulus $0<C_0\le |\xi|\le C^0$, and $\phi$ fulfils 
  \begin{equation}
    \label{eq:xi0}
  \sum_{\nu=-\infty}^\infty \hat\phi_\nu(\xi)=1 \qquad\text{ for $\xi\ne0$},   
  \end{equation}
  then there is a kind of Littlewood--Paley decomposition of every $f\in\cS'(\bR^n)$,
  \begin{equation} \label{LP-id}
    f= \sum_{\nu=-\infty}^\infty \phi_\nu*f \qquad\text{in}\quad \cS'/\cP.
  \end{equation}
The fact that the limit $\nu\to-\infty$ in general gives difficulties was first pointed out by  
Peetre~\cite[p.~52--54]{MR0461123}, who also gave the above remedy that all terms on the two sides must be
understood modulo polynomials (cf.\ Remark~\ref{Peetre-rem}).

However, the statement can be made rather more precise, and we also attempt to explain the
situation from a natural point of view. To achieve this, we say for the
sake of precision that $f\in \cS'$ has \emph{temperate} order $d$, or $\cS'$-order $d$, when
$d$ is the smallest integer such that $f$ is estimated in terms of the seminorm $p_d$: 
\begin{equation} \label{S'-order}
  |\langle f,\psi\rangle|\le c p_d(\psi), \qquad\text{for}\quad p_d(\psi)=\sup_{|\alpha|\le
    d}\|(1+|x|)^dD^\alpha\psi\|_\infty,\quad
   \psi\in\cS.
\end{equation}
The reader may recall that $\Lambda=e^x\cos e^x$ has order $0$ in $\cD'(\bR)$; but $\Lambda$ is
temperate with the value  $\langle \Lambda,\psi\rangle =\int_{\bR}-\psi'(x)\sin e^x\,dx$ for
$\psi\in\cS(\bR)$, so $\Lambda$ has $\cS'$-order $d\ge1$.

For the convergence question in \eqref{LP-id} we may conveniently depart from a convolution 
\begin{align} \label{fPhi-id}
  f-\sum_{\nu=-N}^\infty \phi_\nu*f = \Phi_{-N}*f&= 2^{-Nn}\Phi(2^{-N}\cdot)*f,
\\
  \text{for}\qquad \hat\Phi_{-N}(\xi)&=1-\sum_{\nu=-N}^\infty\hat \phi(2^{-\nu}\xi),
\end{align}
whereby $\hat\Phi_{-N}$ is $C^\infty$ and supported in the ball $|\xi|\le C^0/2^{N+1}$. 
Since $\hat\Phi_{-N}(\xi)=\hat\Phi(2^N\xi)$ holds by inspection if $\Phi:=\Phi_0$,
this is consistent with our notation for dilations.

As a simple example, for $f\in L_1(\bR^n)$ the right-hand side of \eqref{fPhi-id} tends to $0$. In fact, it is
$\cO(2^{-nN})$ as the mere convolution $\Phi(2^{-N}\cdot)*f$ converges to $\Phi(0)\int f\,dx$. So addition of polynomials
in \eqref{LP-id} is unnecessary for such $f$. 

Before analysing \eqref{fPhi-id} for general $f\in\cS'$, we first recall that for any $f\in\cS'$, $\Phi\in\cS$ the convolution
$t^n\Phi(t\cdot)*f$ converges for $t\to\infty$ to $cf$ where
$c=\int\Phi\,dx=\hat\Phi(0)$. Moreover, for $c=0$ the number of vanishing moments of $\Phi$
determines the leading terms and rate of convergence to $0$ for $t\to\infty$, since by Taylor expansion of $\hat\Phi$,
\begin{equation}
  t^n\Phi(t\cdot)*f= \hat\Phi(0) f+\frac1t \sum_{|\alpha|=1}\partial^\alpha\hat\Phi(0)D^\alpha f
   +\dots+\frac1{t^M} \sum_{|\alpha|=M}\frac{\partial^\alpha\hat\Phi(0)}{\alpha!}D^\alpha f+R_M.
\end{equation}
In the ``wrong'' limit $t\to0^+$ the situation is radically different, as convergence cannot be expected (cf.\ $d>n$ below). 
Nevertheless there is an optimal asymptotics formula obtained from the Taylor polynomial
$P_m(x)$ of the $C^\infty$-function $t^n\Phi(t\cdot)*f$ itself:
\begin{equation} \label{t0-id}
  t^n\Phi(t\cdot)*f=P_m+R_m=\sum_{|\alpha|\le m}\partial^\alpha(t^n\Phi(t\cdot)*f)(0)\frac{x^\alpha}{\alpha!}+R_m.
\end{equation}
This asymptotics is elementary in nature, and it may well be folklore. But in lack of a reference
we give a proof of the formula and the optimality. 
It is convenient first to observe the following decomposition of Schwartz functions.

\begin{lemma}  \label{SM-lem}
  When $\psi\in\cS$ has a trivial Taylor polynomial of degree $M\ge0$ at $x=0$, then there are other functions
  $\Psi_\gamma\in\cS$ such that $\psi(x)=\sum_{|\gamma|=M+1} x^\gamma\Psi_\gamma(x)$.
\end{lemma}
\begin{proof}
  For $\psi\in C_0^\infty(\bR^n)$ the claim follows at once from Taylor's formula by
  multiplying both sides by a cut-off function $\chi$ equal to $1$ around $\supp\psi$.
  One can reduce to this case by means of a partition of unity, for when $0\notin\supp\psi$ the
  multinomial formula shows that one can take
  $\Psi_\gamma(x)=(M+1)!x^\gamma\psi(x) |x|^{-2(M+1)}/\gamma! \in\cS$.
\end{proof}

Optimality of \eqref{t0-id} is obtained even among polynomials $Q$ with $t$-dependent degrees:
\begin{proposition}  \label{asymp-prop}
  If $f\in\cS'$ is of $\cS'$-order $d\ge 0$ and $\Phi\in\cS$, 
  the Taylor polynomial $P_m(x)$ of degree $m\ge -1$ satisfies 
  the asymptotics formula \eqref{t0-id}  with terms that are $\cO(t^{n+|\alpha|-d})$ in
  $\cS'$-seminorm for $t\to0^+$,
  whilst the remainder term similarly is
  \begin{equation}
    R_m=\cO(t^{n+m+1-d}).
  \end{equation}
   Any polynomial $Q(x)=\sum_{|\alpha|\le m(t)}c_\alpha(t)x^\alpha$ fulfilling formula \eqref{t0-id} for a remainder $R=o(1)$ in
   $\cS'$-seminorm is given by 
   \begin{equation}
    c_\alpha(t)=\frac1{\alpha!}\partial^\alpha(t^n\Phi(t\cdot)*f)(0)+o(1),   
   \end{equation}
   where  $c_\alpha(t)=o(1)$ for $t\to0^+$  when $|\alpha|>d-n$.
\end{proposition}

\begin{remark}  \label{conv-rem}
The Taylor polynomial $P_m$ is well defined even for $\Phi\in\cS$, due to the well-known fact
that the convolution 
$f\mapsto\langle f,\partial^\alpha\Phi(t(x-\cdot))\rangle$ is continuous $\cS'\to C^\infty$.
\end{remark}

\begin{proof} 
Here $\langle\cdot,\cdot\rangle$ denotes the bilinear form; we take $0<t<1$ and set $\Phi_t=t^n\Phi(t\cdot)$.

In case $m=-1$, i.e.\ $P_m\equiv 0$, the statement is just that $\Phi_t*f=\cO(t^{n-d})$, but
clearly the seminorm $|\langle f, \overline{\tilde\Phi_t*\psi}\rangle|$ is less than
$t^{n-d}c\|(1+|t\cdot|)^d\tilde\Phi(t\cdot)\|_\infty \sum_{|\beta|\le d}\int (1+|\cdot|)^d|D^\beta \psi|\,dx$.

For general $m\ge0$ we estimate $|\langle R_m,\bar\psi\rangle|$ 
by moving the Taylor expansion to the test function $\psi\in\cS$, as 
the two formulas $I=(2\pi)^{-n}\bar\cF\cF$ and $\cF1=(2\pi)^n\delta_0$ give
  \begin{equation}
    \begin{split}
    (2\pi)^n\langle R_m,\bar\psi\rangle &= 
   \langle\hat\Phi(\cdot/t)\hat f,  \overline{\hat\psi}\rangle
   -\sum_{|\alpha|\le m} \frac1{\alpha!}\langle\delta_0,\partial^\alpha\Phi_t*f\rangle_{\cE'\times C^\infty}
   \langle\cF1,\overline{(\operatorname{i}\partial)^\alpha\hat\psi}\rangle
 \\
   &=\langle\hat f,  \hat\Phi(\cdot/t)\overline{\hat\psi}\rangle
   -\sum_{|\alpha|\le m} \langle \hat f,\hat\Phi(\cdot/t)(\operatorname{i}\xi)^\alpha/\alpha!\rangle
   (-\operatorname{i}\partial)^{\alpha}\overline{\hat\psi(0)}
 \\  
  &=\langle\hat f,  \hat\Phi(\cdot/t)\Big(\hat\psi
   -\sum_{|\alpha|\le m} \partial^\alpha
   \hat\psi(0)\xi^\alpha/\alpha!\Big)^{\overline{~}}\,\rangle.
    \end{split}
 \label{Rpsi-id}
  \end{equation}
Indeed, the second line is seen at once for a Schwartz function $f$, and it extends to all
$f\in\cS'$ by density and the continuity in Remark~\ref{conv-rem}.

In \eqref{Rpsi-id} it follows from Lemma~\ref{SM-lem} that the last difference  
has the form $\sum_{|\alpha|=m+1}\xi^\alpha\hat\psi_\alpha(\xi)$
for certain $\hat\psi_\alpha$ in $\cS$.
As the $\cS'$-order of $f$ is $d$,  we get from \eqref{S'-order} 
\begin{equation} \label{R-est}
  |\langle R_m,\psi\rangle|\le c\sum_{|\alpha|=m+1}\sum_{|\beta|\le d}\|(1+|\cdot|)^d\tilde\Phi_t*D^{\alpha+\beta}\psi_\alpha\|_\infty.
\end{equation}
Because of the $\xi^\alpha$, each $D^{\alpha+\beta}\psi_\alpha$ has vanishing moments at least up to order $m$, so  if
we use the uniform estimates in Lemma~\ref{lem:p*p} for $t=2^{\mu}$, cf. Remark~\ref{p*p-rem},
\begin{equation} \label{R-est'}
  \begin{split}
      |(1+|x|)^dt^n\tilde\Phi(t\cdot)*D^{\alpha+\beta}\psi_\alpha(x)|
   &\le
    t^{n-d}|(1+|tx|)^d\tilde\Phi(t\cdot)*D^{\alpha+\beta}\psi_\alpha(x)|
\\ 
   &\le C''_{d,m}t^{n-d+(m+1)}.
  \end{split}
\end{equation}
By summing over $|\alpha|\le m$ it follows that $R_m=\cO(t^{n+m+1-d})$ in seminorm. 

Concerning $\langle\partial^\alpha\Phi_t*f(0)x^\alpha,\overline\psi\rangle$ for $|\alpha|\le m$ we may follow
this term through the calculation \eqref{Rpsi-id} yielding simply $\partial^\alpha\hat\psi(0)/\alpha!$ times
$\langle f,\overline{t^nD^\alpha(\tilde\Phi(t\cdot))}\rangle=t^{n+|\alpha|}\langle f,\overline{D^\alpha\tilde\Phi(t\cdot)}\rangle$.
Using \eqref{S'-order} directly and handling $t$ as in \eqref{R-est'}, all such terms are seen to be $\cO(t^{n+|\alpha|-d})$.

If $\Phi_t*f=Q+o(1)$ for a $t$-dependent polynomial $Q=\sum_{|\alpha|\le m(t)} c_{\alpha}(t)
x^\alpha$, we fix $\alpha$ and take $M\ge|\alpha|$ such that $n+M\ge d$,
whence the above estimate of $R_M$ gives $P_M-Q=O(t)-o(1)=o(1)$ in seminorm for $t\to 0^+$. 
With $\chi\in C_0^\infty$ equal to $1$ around $\xi=0$, we deduce from
$\langle \partial^\beta\delta_0, \xi^\alpha\chi\rangle=\alpha!\delta_{\alpha,\beta}$ that 
\begin{equation}
  o(1)=\langle \hat P_M-\hat Q, \xi^\alpha\chi\rangle
      =(2\pi)^n(-\operatorname{i})^{|\alpha|}(\partial^\alpha\Phi_t*f(0)- c_{\alpha}).
\end{equation}
Thus the coefficient $c_{\alpha}(t)=\frac1{\alpha!}\partial^\alpha \Phi_t*f(0)+o(1)$, i.e.\ 
it must behave asymptotically for $t\to 0^+$ as that of an indiviual term of the Taylor polynomial
$P_M$ at $0$ of $\Phi_t* f$, hence by the previous part of the proof be $o(1)$ if $n+|\alpha|-d>0$. 
\end{proof}

The asymptotic uniqueness of $Q(x)$ in Proposition~\ref{asymp-prop} yields 
\begin{equation}
  \lim_{t\to 0}\partial^\alpha t^n\Phi(t\cdot)*f(0)\ne 0\implies \lim_{t\to0}c_\alpha(t)\ne 0,
\end{equation} 
so even by accepting error terms as vague as $R=o(1)$ there is for
general $f$ no hope to have an \emph{approximating} polynomial $Q$ of degree $m<d-n$.

It should be observed that the estimate $R_m=\cO(t^{n+m+1-d})$ shows that in \eqref{t0-id} the approximation by $P_m(x)$ 
gets increasingly better for $t\to0^+$ if the degree $m$ is fixed so large that $n+m\ge d$. 
For $m=-1$ the convolution itself is $\cO(t^{n-d})$, and goes to $0$ in $\cS'$ if $d<n$, 
in particular it is $\cO(t^n)$ for $d=0$---which for $f\in L_1$ is immediate (cf.\ \eqref{fPhi-id}).
For $d<n$ even all terms in $P_m(x)$ go to $0$ in $\cS'$ for $t\to 0^+$.
 
Furthermore, the convergence rate improves with many vanishing moments:

\begin{corollary} \label{RM-cor}
  When the seminorm $|\langle\cdot,\psi\rangle|$ in Proposition~\ref{asymp-prop} is given by a
  $\psi\in\cS$ having vanishing moments of order $M\ge0$, then
  \begin{equation}
    \langle R_m,\psi\rangle=\cO(t^{n+\max(M,m)+1-d})
  \end{equation}
  and each term in $P_m$ has $\langle\partial^\alpha t^n\Phi(t\cdot)*f,\psi\rangle=0$
  for $|\alpha|\le M$ and else is $\cO(t^{n+|\alpha|-d})$.
\end{corollary}
\begin{proof}
  If $m< M$ the last difference in \eqref{Rpsi-id} just equals $\hat\psi$, which by
  Lemma~\ref{SM-lem} is a sum of terms $\xi^\gamma\hat\psi_\gamma(\xi)$ for $|\gamma|=M+1$.
  Hence the $\alpha$ in \eqref{R-est}--\eqref{R-est'} should be of length $M+1$,
  which as before gives $\cO(t^{n+M+1-d})$.
  The treatment of the individual terms is unchanged, but they clearly vanish for $|\alpha|\le M$ as $\partial^\alpha\hat\psi(0)=0$. 
\end{proof}

Now, if we return to \eqref{fPhi-id}, and take advantage of the fact that $\Phi_{-N}$ is a dilation
as observed there, and if we set 
\begin{equation} \label{PmN-id}
  P_{m,N}(x)= \sum_{|\alpha|\le m} c_{\alpha,N}x^\alpha
  \qquad\text{for $c_{\alpha,N}=\frac{1}{\alpha!}\partial^\alpha 2^{-nN}\Phi(2^{-N}\cdot)*f(0)$},   
\end{equation}
then Proposition~\ref{asymp-prop} gives at once that for each $N\in \bN$, 
\begin{equation} \label{fO-id}
    f=\sum_{\nu=-N}^\infty \phi_\nu*f +P_{m,N}+\cO(2^{-N(n+m+1-d)}).
\end{equation}
For $N\to\infty$ we get for $m=d$ the well-known result of Peetre in \eqref{LP-id}, although the resulting error term
$\cO(2^{-N(n+1)})$ here yields an exponentially fast convergence of the sum $\sum_{\nu=-N}^\infty
\phi_\nu*f$ for $N\to\infty$. This seems to be a novelty in the context.

For $m\ne d$ we furthermore obtain from formula \eqref{fO-id} a general, but sharp version of the homogeneous Littlewood--Paley decomposition:

\begin{proposition} \label{LP-prop}
  If $\phi$ is admissible and fulfils \eqref{eq:xi0} and $f\in\cS'$ is of temperate order $d$, the
  polynomials $P_{m,N}$ in \eqref{PmN-id} satisfy \eqref{fO-id}.
  Moreover any polynomial $P_{m,N}(x)=\sum_{|\alpha|\le m} c_{\alpha,N} x^\alpha$ will fulfil
  \eqref{fO-id} with an $o(1)$-error if and only if its coefficients for $N\to\infty$ satisfy
 \begin{equation}  \label{eq:LWPPo}
   c_{\alpha,N}=\frac1{\alpha!}\partial^\alpha(2^{-nN}\Phi(2^{-N}\cdot)*f)(0) +o(1),
 \end{equation}
where the leading term also is given by $c_{\alpha,N}=\langle \hat
   f,\frac{(\operatorname{i}\xi)^\alpha}{\alpha!}(1-\sum_{v=-N}^\infty\hat\phi_\nu)\rangle /(2\pi)^n$.
\end{proposition}
If desired, one may appeal to the continuity of $D^\alpha$ to obtain, for $m+n\ge d$,
\begin{equation}  \label{Df-id}
   D^\alpha f=\lim_{N\to\infty}(\sum_{\nu=-N}^\infty \phi_\nu*D^\alpha f-D^\alpha P_{m,N}). 
\end{equation}

\begin{remark}
  Kyriazis~\cite{MR1981430} made a study of the Littlewood--Paley decomposition \eqref{LP-id} in
  the space $(\cS_M)'$, i.e.\ the dual of $\cS_M=\{\,\psi\in\cS\mid\int
  x^\alpha\psi \,dx=0 \text{ for } |\alpha|\le M\,\}$. Inspired by this, let us note that by testing against such special
  $\psi$, the remainder term in \eqref{fO-id} improves in view of Corollary~\ref{RM-cor} to $\cO(2^{-N(n+\max(M,m)+1-d)})$.
\end{remark}
\begin{remark}
  Kyriazis~\cite{MR1981430} gave an example of distributions $f_s$, $s>0$, in $\cS'(\bR)$ for which the series \eqref{fO-id} 
  only converges for $m\ge0$; i.e.\ addition of at least constants is necessary to have convergence to $f_s$ in $\cS'(\bR)$.
\end{remark}
\begin{remark} \label{Peetre-rem}
Peetre~\cite[p.~54]{MR0461123} treated convergence in $\cS'$ of the Littlewood--Paley decomposition
\eqref{LP-id}, 
which he essentially stated with unspecified polynomials $P$ and $P_N$ in the form
\begin{equation} \label{Peetre-id}
  f-P=\lim_{N\to\infty}(\sum_{\nu=-N}^\infty \phi_\nu*f-P_N) .
\end{equation}
According to Proposition~\ref{LP-prop} subtraction of $P$ can be avoided (of
course $P$ can also be added on both sides, replacing $P_N$ by $P_N-P$, to have only $f$ on the left).
Peetre sketched a proof based on the obvious convergence of the differentiated series
$\sum_\nu \phi_\nu*D^\alpha f$ for $|\alpha|=d+1$, where polynomials of degree $d$ form the common null space of
these $D^\alpha$, leading to the $P_N$ and $P$---but no details were given on convergence in \eqref{Peetre-id}. 
Frazier and Jawerth~\cite{MR808825} claimed restrictions on the degrees of $P_N$ and $P$.  
Later Kyriazis~\cite{MR1981430} gave a full proof of \eqref{Peetre-id}, and so did Bownik and Ho~\cite{MR2186983}.
Our Proposition~\ref{LP-prop} presents an alternative approach with Taylor polynomials
$P_{m,N}$ of an \emph{arbitrary} degree $m\ge-1$, which by the asymptotic uniqueness in Proposition~\ref{asymp-prop}
yields that the above $P_N$ and $P$ must be \emph{interrelated} by $P_N-P=P_{m,N}+ o(1)$. 
It also provides a comprehensive error analysis, entailing that the band-limited series
$\sum_{\nu\ge-N}\phi_\nu*f$ plus $P_{m,N}$ converges to $f$ itself
in the topology of $\cS'$ whenever $m\ge d-n$.
\end{remark}

\section*{Acknowledgement}
We wish to thank the two anonymous reviewers for their careful reading of the manuscript,
which impelled us to insert Remark~\ref{emb-rem}, Remark~\ref{SBE-rem} and Corollary~\ref{Tpsi-cor}
and a number of minor changes to improve the presentation in the final version.


\end{document}